\documentclass[11pt,reqno]{amsart}
\usepackage[utf8]{inputenc}

\usepackage{faktor,amsfonts,amsmath,amscd,amssymb,amsbsy,amsthm,amstext,amsopn,amsxtra,mathrsfs,stmaryrd,graphicx,mathtools}
\usepackage[doi=false,isbn=false,eprint=false,maxbibnames=9,maxcitenames=50]{biblatex}
\DeclareFieldFormat{url}{Available at\addcolon\space\url{#1}} 
\usepackage{extarrows}
\usepackage{comment}
\usepackage{tikz-cd}
\usepackage{booktabs} 
\usepackage{url}

\usepackage{breakurl}
\usepackage[breaklinks]{hyperref}
\tikzset{tail reversed/.code={\pgfsetarrowsstart{tikzcd to}}}
\tikzset{2tail/.code={\pgfsetarrowsstart{Implies[reversed]}}}
\tikzset{2tail reversed/.code={\pgfsetarrowsstart{Implies}}}
\tikzset{no body/.style={/tikz/dash pattern=on 0 off 1mm}}

\newcommand{\CC}{\mathbb{C}}

\newcommand{\RR}{\mathbb{R}}

\newcommand{\QQ}{\mathbb{Q}}

\newcommand{\Pic}{\mathrm{Pic}}
\newcommand{\oh}{\mathcal{O}}
\newcommand{\Spec}{\mathrm{Spec}}
\newcommand{\Proj}{\mathrm{Proj}}
\newcommand{\HH}{\mathrm{H}}
\newcommand{\VV}{\mathbb{V}}
\newcommand{\ZZ}{\mathbb{Z}}

\newcommand{\Af}{\mathbb{A}}
\newcommand{\Rep}{\mathrm{Rep}}
\newcommand{\Sym}{\mathrm{Sym}}
\newcommand{\PP}{\mathbb{P}}
\newcommand{\rar}{\rightarrow}

\newcommand{\xrar}[1]{\xrightarrow{#1}}

\newcommand{\Cox}{\mathrm{Cox}}
\newcommand{\Hom}{\mathrm{Hom}}
\newcommand{\git}{\ensuremath{\operatorname{\!/\!\!/\!}}}
\newcommand{\GL}{\mathrm{GL}}
\newcommand{\SL}{\mathrm{SL}}

\newcommand{\Coker}{\mathrm{Coker}}

\newcommand{\MM}{\mathfrak{M}}
\newcommand{\calX}{\mathcal{X}}
\newcommand{\calY}{\mathcal{Y}}
\newcommand{\calZ}{\mathcal{Z}}
\newcommand{\calS}{\mathcal{S}}

\newcommand{\calP}{\mathcal{P}}
\newcommand{\tot}{\mathrm{tot}}
\newcommand{\cone}{{\mathrm{Cone}}}
\newcommand{\free}{\mathrm{free}}

\newcommand{\calW}{\mathcal{W}}

\newtheorem{theorem}{Theorem}[section]
\newtheorem{lemma}[theorem]{Lemma}
\newtheorem{proposition}[theorem]{Proposition}
\newtheorem{corollary}[theorem]{Corollary}

\theoremstyle{definition}
\newtheorem{definition}[theorem]{Definition}
\newtheorem{example}[theorem]{Example}

\newtheorem{notation}[theorem]{Notation}
\newtheorem*{acknowledgements}{Acknowledgements}

\theoremstyle{remark}
\newtheorem{remark}[theorem]{Remark}

\usepackage{geometry}
\title{All crepant resolutions of hyperpolygon spaces via their Cox rings.}
\author{Austin Hubbard}
\thanks{Email: \href{mailto:ah2941@bath.ac.uk}{ah2941@bath.ac.uk}, Affiliation: University of Bath}
\date{}
\email{\href{mailto:ah2941@bath.ac.uk}{ah2941@bath.ac.uk}}
\address{University of Bath}

\begin{document}
\begin{abstract}
    We construct and enumerate all crepant resolutions of hyperpolygon spaces, a family of conical symplectic singularities arising as Nakajima quiver varieties associated to a star-shaped quiver. We provide an explicit presentation of the Cox ring of any such crepant resolution. Using techniques developed by Arzhantsev-Derenthal-Hausen-Laface we construct all crepant resolutions of the hyperpolygon spaces, including those which are not projective over the singularity. We find that the number of crepant resolutions equals the Hoşten-Morris numbers. 
    In proving these results, we obtain a description of all complete geometric quotients associated to the classical GIT problem constructing moduli spaces of ordered points on the projective line. These moduli spaces appear as the Lagrangian subvarieties of crepant resolutions of hyperpolygon spaces fixed under the conical action.
\end{abstract}
\maketitle
\tableofcontents

\section{Introduction}
Hyperpolygon spaces are Nakajima quiver varieties associated to an $n$-pointed star-shaped quiver for a fixed dimension vector $v$, for $n\geq 4$. They form a family of symplectic varieties of dimension $2n-6$, generalising the crepant partial resolutions of the Kleinian singularity of type $D_4$. First introduced by Konno \cite{Konno} as hyperkähler analogues of polygon spaces, hyperpolygon spaces have since been investigated by Harada-Proudfoot \cite{HP} and Godinho-Mandini \cite{GodinhoMandini}. 
Polygon spaces arise as compactifications of the moduli space of $n$-distinct ordered points on $\PP^1$ and have served as foundational examples in the theory of GIT. Much like the well-studied polygon spaces, hyperpolygon spaces have provided a ripe testing ground for conjectures pertaining to Nakajima quiver varieties and algebraic symplectic reduction.

We provide explicit presentations of the Cox rings of smooth hyperpolygon spaces $\MM_\theta$, for $n\geq 5$. As an application, we construct all crepant resolutions of the most singular hyperpolygon spaces $\MM_0$. Work of Bellamy-Craw-Rayan-Schedler-Weiss \cite{BCSRW} establishes a construction of all projective crepant resolutions of $\MM_0$ as Nakajima quiver varieties, using Mumford's GIT. To produce more general quotients we use the $A_2$-quotient construction developed by Arzhantsev-Derenthal-Hausen-Laface \cite{ADHL}.
This allows the construction of crepant resolutions that are not projective over $\MM_0$, which have been difficult to study until now.

We introduce maximally-biconnected complexes, a combinatorial object which we show parameterises the crepant resolutions of $\MM_0$. The number of maximally-biconnected complexes on the set $\{1,\ldots,n\}$ is given by the $n$-th Hoşten-Morris number \cite{HM}. We find that the number of crepant resolutions of $\MM_0$ up to isomorphism over $\MM_0$ equals $\lambda(n)$.

Any crepant resolution of $\MM_0$ admits a $\CC^*$-action extending the conical $\CC^*$-action on $\MM_0$. There is a unique Lagrangian subvariety inside the $\CC^*$-fixed locus, which is in almost all cases a polygon space. We see that the small birational transformation relating a pair of crepant resolutions is completely determined by the induced transformation on this Lagrangian subvariety, providing a close link between the crepant resolutions of $\MM_0$ and polygon spaces. We study the polygon space GIT problem of $n$ ordered points on $\PP^1$, extending the classical results to include non-projective quotients. This GIT problem is given by the action of $T:=(\CC^*)^{n}$ on the affine variety $Y$ given as the spectrum of the Cox ring of $\PP^{n-3}$ blown-up at $n-1$ general points, for $n\geq 5$. We provide a combinatorial characterisation of the open subsets $U\subseteq Y$ which admit geometric quotients $U\rar U\git T$ with $U\git T$ complete. We show that the number of such open subsets equals $\lambda(n)-n$.

\subsection{A presentation of the Cox ring.}
The hyperpolygon spaces $\MM_\theta$ are constructed as GIT quotients $\MM_\theta:=\mu^{-1}(0)\git_\theta G$, dependent on a choice of parameter $\theta\in \Theta\cong G^\vee\otimes_\ZZ \QQ$. The most singular hyperpolygon space $\MM_0$, given as the affine GIT quotient $\mu^{-1}(0)\git G$, is a conical symplectic singularity (see Definition \ref{def:consymsin}). It is shown by Bellamy-Craw-Rayan-Schedler-Weiss \cite{BCSRW}, that for generic $\theta\in \Theta$, the natural morphism $\pi_\theta:\MM_\theta \rar \MM_0$, induced by varying the parameter in $\Theta$, is a projective crepant resolution (see Theorem \ref{thm:BCSRW}). By work of Namikawa \cite{Namikawa}, for generic $\theta\in \Theta$, $\MM_\theta$ is a Mori dream space (see Definition \ref{def:MDS}), i.e.\ it has finitely-generated Cox ring. Even in cases where the Cox ring of a variety is finitely-generated, it is in general a difficult problem to provide explicit generators and relations for it.

The case of $n=4$ is extremely well-studied. The graph underlying the star-shaped quiver with $4$ external vertices is the affine Dynkin diagram of type $\widetilde{D}_4$. By work of Kronheimer \cite{Kronheimer}, the affine quotient $\MM_0$ is the Kleinian singularity of type $D_4$. For $\theta\in\Theta$ generic, $\pi_\theta:\MM_\theta \rar \MM_0$ is the minimal resolution of $\MM_0$. In this case, the Cox ring of the minimal resolution of $\MM_0$ is computed by Donten-Bury \cite[Example~6.14]{DontenBury}.

In order to describe the Cox ring of $\MM_\theta$, we introduce auxilliary geometric and algebraic objects (see Definition \ref{def:calX}): the affine variety $\mathcal{X}$ is the affinisation of a principal bundle over $\MM_\theta$; the $\CC$-algebra $R$ is a $\ZZ^n$-graded polynomial ring, and $I\subseteq R$ is a homogeneous ideal.

\begin{theorem}
    Fix $n\geq 5$. The following $\ZZ^n$-graded $\CC$-algebras are isomorphic
    \begin{enumerate}
        \item $\Cox(\mathfrak{M}_\theta)$, for $\theta\in\Theta$ generic.
        \item $\CC[\calX]^{\SL(2)}$.
    \item $R/I$.
    \end{enumerate}
\end{theorem} 
The variety $\mathcal{X}$ is a complete intersection of codimension 3 in $\Af^{3n}$ with a particularly simple description, suggesting that if one wants to work with a coordinate ring for $\MM_\theta$ defined by simple equations, then $\CC[\calX]$ can be seen as a better candidate than $\Cox(\MM_\theta)$.
\subsection{All crepant resolutions of $\MM_0$.}
In \cite{BCSRW}, it is shown that any projective crepant resolution of $\MM_0$ arises as a morphism $\pi_\theta:\MM_\theta \rar \MM_0$, for generic $\theta$ in a unique GIT chamber within the positive orthant $F\subseteq\Theta$. The positive orthant $F$ is a fundamental domain for the action of the Namikawa-Weyl group on $\Theta$. They further provide an explicit hyperplane arrangement $\mathcal{A}$ which determines the GIT fan for the action of $G$ on $\mu^{-1}(0)$ (see discussion above Definition~\ref{def:hyperpolygon}). This gives a combinatorial method for computing the number of projective crepant resolutions: it is the number of chambers of the GIT fan contained in $F$.

However, a crepant resolution need not be a projective morphism. We provide a construction of all the crepant resolutions of $\MM_0$, including non-projective resolutions. As an application, we show that for $n\geq 6$, there are many examples of non-projective crepant resolutions of $\MM_0$. Each crepant resolution arises as the geometric quotient of an open subset of $X:=\Spec(\Cox(\MM_\theta))$ under the natural torus action induced by the $\Pic(\MM_\theta)$-grading. We utilise the construction of $A_2$-quotients via bunches of orbit cones given by Arzhantsev-Derenthal-Hausen-Laface \cite[Section~3.1]{ADHL}. The construction takes as input a bunch of orbit cones $\Psi\subseteq\Omega_X$ (see Definition \ref{def:bunch}), and produces an open subset $U(\Psi)\subseteq X$ admitting a good quotient $U(\Psi)\rar U(\Psi)\git T$.
We define maximally-biconnected complexes on $[n]:=\{1,\ldots,n\}$ (see Definition \ref{def:complex}), and describe a map associating to every such complex $\Delta$ a maximal bunch of orbit cones $\Psi_\Delta$.
\begin{theorem}\label{thm:A}
    Fix $n\geq 5$. Let $\mathscr{C}$ denote the set of crepant resolutions $\pi:\MM\rar \MM_0$ up to isomorphism over $\MM_0$.
    There is a bijection
    \begin{align*}
        \{\text{Maximally-biconnected complexes on }[n]\} &\longrightarrow \mathscr{C}, \\
        \Delta &\longmapsto \left[ U(\Psi_\Delta)\git T \rar \MM_0. \right]
    \end{align*}
\end{theorem}
\begin{corollary}
    We have $\#\mathscr{C} = \lambda(n)$, where $\lambda(n)$ denotes the $n$-th Hoşten-Morris number (see Definition \ref{def:hm}). 
\end{corollary}
The Hoşten-Morris numbers $\lambda(n)$ are known for $n\leq 9$. The following table gives a comparison between the number $P(n)$ of projective crepant resolutions of $\MM_0$ (see \cite[Remark~3.9]{BCSRW}) and the number $\lambda(n)$ of all crepant resolutions.
\medskip
    \begin{center}
        \begin{tabular}{l|ccccccl}\toprule
            \ \ $n$ & 5 & 6 & 7 & 8 & 9 \\
            $P(n)$ & 81 & 1684 & 122921 & 33207256 & 34448225389 \\
            $\lambda(n)$ & 81 & 2646 & 1422564 & 229809982112 & 423295099074735261880  \\ \bottomrule
        \end{tabular}
    \end{center}    
\medskip
In the case $n=5$, we see that $\MM_0$ has 81 crepant resolutions, all of which are projective over $\MM_0$. In \cite[Theorem~1.1]{BCSRW} it is shown that $\MM_0\cong \CC^4/\Gamma$, where $\Gamma\subseteq \mathrm{Sp}(4)$ is a fixed group of order 32 (see Example \ref{ex:n=5}). This links the construction of all projective crepant resolutions of $\MM_0$ given in \cite{BCSRW} with the construction of 81 projective crepant resolutions of $\CC^4/\Gamma$ given by Donten-Bury-Wi\'{s}niewski \cite{DBW}.
\subsection{All polygon spaces.}
Polygon spaces can be constructed as GIT quotients of an affine variety $Y$ isomorphic to the affine cone over the Grassmannian $\mathrm{Gr}(2,n)$, by an action of $T:= (\CC^*)^n$ (see Definition \ref{def:polygon}). We define full maximally-biconnected complexes on $[n]$ (see Definition \ref{def:complex}), and describe a map associating to every such complex $\Delta$, a maximal bunch of orbit cones $\Phi_\Delta\subseteq\Omega_Y$.
\begin{theorem}\label{thm:C}
    Fix $n\geq 5$. Let $\mathscr{U}$ denote the set of $T$-invariant open subschemes $U\subseteq Y$ admitting a geometric quotient $U\rar U\git T$, such that $U\git T$ is complete. There is a bijection
        \begin{align*}
            \{\text{Full maximally-biconnected complexes on }[n]\}&\longrightarrow \mathscr{U},\\
            \Delta &\longmapsto U(\Phi_\Delta)/T.
        \end{align*}
\end{theorem}
\begin{corollary}
    We have $\#\mathscr{U} = \lambda(n)-n$, where $\lambda(n)$ denotes the $n$-th Hoşten-Morris number (see Definition \ref{def:hm}).
\end{corollary}
For each hyperpolygon space $\MM_\theta$, there is a $\CC^*$-action contracting $\MM_\theta$ into the zero fiber $\pi_\theta^{-1}(0)$. If $\theta$ is generic, there is a unique Lagrangian subvariety fixed under this action which is almost always a polygon space (see Definition \ref{def:polygon}). If $\MM$ is an arbitrary crepant resolution of $\MM_0$, there remains a contracting $\CC^*$ action and a unique Lagrangian subvariety fixed under this action, which may not be projective.

The quotients arising from Theorem \ref{thm:C} describe the $\CC^*$-fixed Lagrangians inside arbitrary crepant resolutions $\MM$ of $\MM_0$. In fact, Theorem \ref{thm:C} is a vital component in the proof of Theorem \ref{thm:A}. It follows from the constructions involved in Theorems $\ref{thm:A}$ and $\ref{thm:C}$ that a crepant resolution $\MM$ is projective if and only if the $\CC^*$-fixed Lagrangian it contains is projective. This explains the fact that for $n=5$, in which case $\dim(\MM_0)=4$, all crepant resolutions $\MM$ are projective, since the $\CC^*$-fixed Lagrangian in $\MM$ is a smooth surface, which must be projective.

As an example, let $Z_3$ be the projective 3-fold in $\PP^5$ defined by the equations
    \begin{align*}
        \sum_{i=0}^5 x_i = 0, \ \ \sum_{i=0}^5 x_i^3 =0.
    \end{align*}
    This is the Segre cubic; it is constructed as a projective GIT quotient of $Y$ (see Example \ref{ex:segre}). Its singular locus consists of 10 ordinary double points, each of which locally admits a pair of crepant resolutions. Since each singular point can be resolved independently, $Z_3$ has $2^{10}$ crepant resolutions. Work of Finkelnberg \cite{Finkelnberg} establishes that only 332 of them are projective. All $2^{10}$ minimal resolutions arise as quotients of $Y$ described in Theorem \ref{thm:C}.

     There is a 5-dimensional analogue of the Segre cubic, which we denote by $Z_5$, constructed as a projective GIT quotient of $Y$ (see Example \ref{ex:segre2}). Its singular locus consists of 35 ordinary double points. Thus $Z_5$ admits $2^{35}$ crepant resolutions, of which only 495504 are projective (see Corollary \ref{cor:morichamber}). All minimal resolutions can again be constructed as quotients of $Y$ described by Theorem \ref{thm:C}.

\begin{acknowledgements}
    I would like to thank my supervisor Alastair Craw for introducing me to this problem and providing much helpful feedback. I'm grateful to Travis Schedler and Nick Proudfoot for their helpful commments.
    The author was supported fully by Research Project Grant RPG-2021-149 from The Leverhulme Trust.
\end{acknowledgements}
\section{Background}
\subsection{GIT quotients.}
We begin with a summary of various GIT quotient constructions, following the development given by Arzhantsev-Derenthal-Hausen-Laface \cite{ADHL}.

A variety is an integral separated scheme of finite-type over $\CC$. Let $G$ be a reductive algebraic group acting on an affine variety $X=\Spec(R)$. Finite generation of the subring of $G$-invariant functions $R^G\subseteq R$ provides the construction of an affine quotient variety $X\git G = \Spec(R^G)$.

There is an induced $G$-action on the coordinate ring $R$. The $R^G$-module of semi-invariant functions of weight $\theta\in \Theta$ is
\[ R_\theta:=\{f\in R \ \vert \ g\cdot f = \theta(g^{-1})f\}\subseteq R.\]
Let $\Theta:=G^\vee\otimes_{\ZZ}\QQ$. If $\theta\in \Theta\setminus G^\vee$, set $R_\theta=0$.
Given $\theta\in \Theta$, Mumford's GIT \cite{GIT} gives a construction of refined quotient spaces $X\git_\theta G$, projective over $X\git G$.
\begin{definition}
    Given $\theta\in \Theta$, a point $x\in X$ is \emph{$\theta$-semistable} if there exists some $n>0$, and some homogeneous $f\in R_{n\theta}$, such that $f(x)\neq 0$. Let $X^{\theta\text{-ss}}$ denote the $G$-invariant open subset of $\theta$-semistable points in $X$.
\end{definition}
For all $\theta\in \Theta$, the $\theta$-semistable locus admits a good quotient given by
\[ X\git_\theta G:= \Proj\Big(\bigoplus_{n\geq 0}R_{n\theta}\Big).\]
The inclusion of $R_0$ into $\bigoplus_{n\geq 0} R_{n\theta}$ induces a projective morphism $X\git_\theta G \rar X\git G$.

From now on we consider the action of an algebraic torus $T$ on $X$. In this case the coordinate ring $R$ is graded by $T^\vee$. We have $R\cong \bigoplus_{\theta\in T^\vee}R_\theta$. Suppose $f\in R_\theta$ for some $\theta\in T^\vee\subseteq\Theta$, then let $\deg(f)=\theta$.
\begin{definition}
	The \emph{orbit cone} of $x\in X$ is the convex polyhedral cone 
	\[ \omega_x := \QQ_{>0}\cdot\{\theta\in \Theta \ \vert \ f(x)\neq 0 \ \text{for some} \ f\in R_\theta\}
    \subseteq\Theta.\]
    We denote by $\Omega_X$ the set of all orbit cones of points in $X$. This is a finite set \cite[Proposition~3.1.1.10]{ADHL}. Given an orbit cone $\omega\in \Omega_X$, we denote its relative interior by $\omega^\circ$.
    The \emph{GIT cone} associated to $\theta\in \Theta$ is the intersection
    \[ C(\theta) := \bigcap_{\{\omega_x\in \Omega_X \vert \theta\in \omega_x\}}\omega_x.\]
    We say $\theta$ is \emph{generic} if $C(\theta)$ is top dimensional. We call the interior of a top dimensional GIT cone a GIT chamber.
\end{definition}
\begin{definition}
    A \emph{quasi-fan} in $\Theta$ is a finite collection of convex polyhedral cones in $\Theta$, closed under taking faces, and such that the intersection of any two cones is a face of both. A \emph{fan} is a quasi-fan consisting of strongly-convex polyhedral cones.
\end{definition}
\begin{theorem}[\cite{ADHL}, Theorem 3.1.2.8]
    The collection of all GIT cones gives a quasi-fan in $\Theta$. Given $\theta_1,\theta_2\in \Theta$, we have that $X^{\theta_1\text{-ss}}=X^{\theta_2\text{-ss}}$ if and only if $C(\theta_1)=C(\theta_2)$.
\end{theorem}

We are interested in constructing quotients that are not projective over $X\git T$. Since the natural morphism $X\git_\theta T\rar X\git T$ is always projective we require a construction that produces a larger class of quotients.
\begin{definition}\label{def:bunch}
    A \emph{bunch of orbit cones} is a non-empty collection of orbit cones $\Phi\subseteq \Omega_X$, such that:
    \begin{enumerate}
        \item Given $\omega_1,\omega_2 \in \Phi$, then $\omega_1^\circ \cap \omega_2^\circ \neq \emptyset$.
        \item Given $\omega_1\in \Phi$, and $\omega_2\in \Omega_X$ such that $\omega_1^\circ\subseteq \omega_2^\circ$, then $\omega_2\in \Phi$.
    \end{enumerate}
A \emph{maximal bunch of orbit cones} is a bunch of orbit cones that is not contained in any strictly larger bunch of orbit cones.
\end{definition}
\begin{definition}
    Given a bunch of orbit cones $\Phi\subseteq \Omega_X$, we define an open subscheme
\[ U(\Phi) = \{x\in X \ \vert \ \omega_x\in \Phi\}.\]    
\end{definition}
\begin{definition}
    A variety $Y$ has the \emph{$A_2$-property} if for all $y_0,y_1\in Y$ there exists an affine open subset of $Y$ containing both $y_0$ and $y_1$.
\end{definition}
\begin{proposition}[\cite{ADHL}, Proposition 3.1.3.8]\label{prop:3138}
    Let $\Phi\subseteq \Omega_X$ be a bunch of orbit cones. Then $U(\Phi)$ is a $T$-invariant open subscheme admitting a good quotient $U(\Phi)\rar U(\Phi)\git T$. The quotient $U(\Phi)\git T$ has the $A_2$-property.
\end{proposition}
Let $x\in X$, and let $f_1,\ldots,f_r\in R$ be a collection of homogeneous functions such that $f_i(x)\neq 0$ and $\omega_x$ is generated by $\deg(f_i)$. We define a $T$-invariant open subscheme $U_x:=\Spec(R[f_1^{-1},\cdots,f_r^{-1}])\subseteq U(\Phi)$. This open subscheme does not depend on the choice of $f_i$ \cite[Construction~3.1.3.7]{ADHL}. Since $x\in U_x$, we have that the collection of $U_x$ for $x\in U(\Phi)$ forms an open covering of $U(\Phi)$. The affine quotients $U_x\git T$ glue together appropriately to form a good quotient.
\begin{remark}\label{rmk:projcones}
    Given $\theta\in \Theta$, define $\Phi_\theta:=\{ \omega\in\Omega_X \ \vert \ \theta\in \omega^\circ\}$. This is a maximal bunch of orbit cones, and $U(\Phi_\theta)=X^{\theta\text{-ss}}$. Hence the set of quotients defined via bunches of orbit cones includes the set of projective GIT quotients.
\end{remark}
\begin{definition}
    Let $U\subseteq X$ be a $T$-invariant open subset admitting a good quotient $\pi:U\rar U\git T$. Given $V\subseteq U$, we say $V$ is \emph{saturated} in $U$ if $V=\pi^{-1}(\pi(V))$.
\end{definition}

    Let $\mathscr{A}$ denote the collection of $T$-invariant open subschemes $U\subseteq X$ such that $U$ admits a good quotient $U\rar U\git T$, where $U\git T$ has the $A_2$-property, and $U$ is maximal with respect to inclusion of $T$-invariant saturated open subsets admitting good quotients.
\begin{theorem}[\cite{ADHL}, Theorem 3.1.4.4]
    There is a bijection 
    \begin{align*} \{\text{Maximal bunches of orbit cones }\Phi\subseteq\Omega_X\} &\longrightarrow \mathscr{A}, \\
        \Phi &\longmapsto U(\Phi).
    \end{align*}
\end{theorem}
\begin{remark}
    For general torus actions on affine varieties, there may exist open subschemes admitting good quotients such that the quotient spaces do not have the $A_2$-property, see \cite[Exercise~3.8]{ADHL} for an example.
\end{remark}
\subsection{Cox rings and Mori dream spaces.}
In this paper, we consider complete varieties and resolutions of singular affine varieties, both of which fit into the following class.
\begin{definition}
    A variety $X$ is called \emph{proper (projective) over affine} if $\Gamma(X,\oh_X)$ is a finitely generated $\CC$-algebra, and the natural affinisation morphism 
    $X\rar \Spec(\Gamma(X,\oh_X))$ is proper (projective).
\end{definition}

\begin{definition}\label{def:coxring}\label{def:MDS}
    Let $X$ be a smooth, proper over affine variety, with finitely generated and free Picard group of rank $n$. Let $L_1,\ldots,L_n$ be a collection of line bundles chosen such that their isomorphism classes form a basis of $\Pic(X)$. Given $\lambda=(\lambda_1,\ldots,\lambda_n)\in \ZZ^n$, let $L_{\lambda}:= L_1^{\otimes \lambda_1}\otimes\cdots\otimes L_n^{\otimes \lambda_n}$.
    The \emph{Cox ring} of $X$ is the $\Pic(X)$-graded $\CC$-algebra given by
\[ \mathrm{Cox}(X) = \bigoplus_{\lambda\in \ZZ^n}\Gamma(X, L_{\lambda}).\]    
    In the case that $\Cox(X)$ is a finitely-generated $\CC$-algebra, then $X$ is a \emph{Mori dream space}. When $X$ is a Mori dream space the \emph{total coordinate space} of $X$ is $\overline{X}:=\Spec(\Cox(X))$.
\end{definition}

Suppose $X$ is a Mori dream space. Then $X$ carries an ample line bundle $L$, satisfying
\[ X\cong \Proj\Big(\bigoplus_{k\geq 0}\Gamma(X,L^{\otimes k})\Big).\]
Let $T$ be the algebraic torus with character group $T^\vee=\Pic(X)$. Since $\Cox(X)$ is $\Pic(X)$-graded, the total coordinate space $\overline{X}:=\Spec(\Cox(X))$ is equipped with a $T$-action. Letting $\theta=[L]\in T^\vee$, we have that $\overline{X}\git_\theta T \cong X$. Hence a Mori dream space is recovered from its total coordinate space as a GIT quotient.

\begin{definition}
    Let $X$ and $L_1,\ldots,L_n$ be as in Definition \ref{def:coxring}. The \emph{Cox sheaf} of $X$ is the $\Pic(X)$-graded sheaf of $\oh_X$-algebras given by
    \[\mathscr{C}ox(X) \vcentcolon= \bigoplus_{\lambda\in\ZZ^n}L_\lambda.\]
    The relative spectrum $\widehat{X}=\Spec_X(\mathscr{C}ox(X))$ carries a $T$ action induced by the $\Pic(X)$-grading on $\mathscr{C}ox(X)$. The \emph{total coordinate space} of $\overline{X}\vcentcolon=\Spec(\Cox(X))$ is the affinisation of $\widehat{X}$.
\end{definition}
\begin{lemma}[\cite{ADHL}, Construction~1.6.3.1]\label{lem:coxsheaf}
    Let $X$ be as in Definition \ref{def:coxring}. There is a $T$-equivariant open immersion $\widehat{X}\rar~\overline{X}$, with complement of codimension $\geq 2$. If $X$ is projective over affine and $L$ is an ample line bundle with $\theta=[L]$, then $\widehat{X}=\overline{X}^{\theta\text{-ss}}$.
\end{lemma}

Let $D$ be an effective divisor on $X$. It follows from finite generation of the Cox ring that the section ring $R_D:=\bigoplus_{n\geq 0}\Gamma(X,\oh(D))$ is finitely generated. 
\begin{definition}
    Given effective divisors $D_1$ and $D_2$ on $X$, there are induced rational maps $\phi_1:X \dashrightarrow \Proj(R_{D_1})$ and $\phi_2:X\dashrightarrow \Proj(R_{D_2})$. We say $D_1$ and $D_2$ are \emph{Mori equivalent} if there exists some $m>0$ and an isomorphism $\psi:\Proj(R_{mD_1})\rar \Proj(R_{mD_2})$, such that $\psi\circ\phi_1=\phi_2$. A \emph{Mori cone} in $\Pic(X)_\QQ$ is the closure of a Mori equivalence class, a \emph{Mori chamber} is the interior of a top-dimensional Mori cone. Since there are natural isomorphisms $\Proj(R_D)\cong \Proj(R_{mD})$, for $m>0$, we see that Mori cones are indeed cones.
\end{definition}
\begin{theorem}[\cite{HuKeel}, Theorem 2.3]\label{thm:hukeel}
    Let $X$ be a Mori dream space. The GIT cones associated to the action of $T$ on $\overline{X}$ are exactly the Mori cones of $X$. 
\end{theorem}
Suppose $\theta = [\oh_X(D)]\in \Pic(X)$, from the definition of the Cox ring we see clearly that $(X\dashrightarrow \overline{X}\git_\theta T) = (X\dashrightarrow \Proj(R_D))$.
\subsection{Polygon spaces.}
In this subsection we focus on a particular family of affine varieties with torus action, constructed as spaces of quiver representations. Fix $n\geq 5$.
\begin{definition}
A quiver $Q$ consists of the data of a vertex set $Q_0$, a set of arrows $Q_1$ and a pair of orienting source and target maps $s,t:Q_1\rar Q_0$. We define a representation of a quiver $Q$ as the data of a vector space $V_i$ assigned to each vertex $i\in Q_0$, together with a linear map $f_a\in \Hom(V_{s(a)},V_{t(a)})$ for each arrow $a\in Q_1$. A dimension vector $v:Q_0\rar \ZZ_{\geq 0}$ assigns to each vertex a non-negative integer. For a fixed dimension vector $v$ we write
\[ \Rep(Q,v) := \{\text{Representations }((V_i)_{i\in Q_0},(f_a)_{a\in Q_1})\ \vert \ \dim(V_i)=v(i)\}.\]   
\end{definition} 
Fixing vector spaces $V_i$ of dimension $v(i)$, for all $i\in Q_0$, we obtain
\[ \Rep(Q,v) \cong \bigoplus_{a\in Q_1}\Hom(V_{s(a)},V_{t(a)}).\]
From now on, we always assume such a fixed collection of vector spaces.

The group $\GL(Q,v):=\prod_{i\in Q_0}\GL(V_i)$ acts on $\Rep(Q,v)$ via simultaneous change of basis:
\[ (g_i)_{i\in Q_0}\cdot (f_a)_{a\in Q_1} = (g_{t(a)}\cdot f_a \cdot g_{s(a)}^{-1})_{a\in Q_1}.\]
The action of $\GL(Q,v)$ on $\Rep(Q,v)$ has global stabiliser $\Delta = \{(\lambda I_i)_{i\in Q_0} \ \vert \ \lambda\in \CC^*\}$, where $I_i$ denotes the identity endomorphism on $V_i$. When working with a fixed quiver $Q$ and dimension vector $v$ we write $G:=\GL(Q,v)/\Delta$.
\begin{definition}\label{def:polyquiver}
The star-shaped quiver $Q$ is the quiver with vertex set $Q_0=\{0,1,\ldots,n\}$ and arrow set $Q_1=\{a_1,\ldots,a_n\}$, with $s(a_i)=0$ and $t(a_i)=i$, for $i=1,\ldots,n$. From now on we fix the dimension vector $v$ with $v(0)=2$, and $v(i)=1$, for $i=1,\ldots,n$.
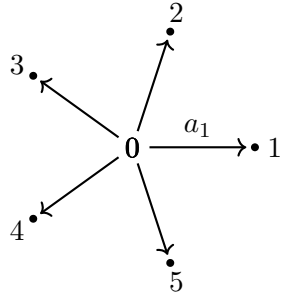
\begin{figure}[ht]
    \centering
    \begin{tikzpicture}[scale=0.75]
            \foreach \x in {1,2,3,4,5} {
                \begin{scope}[rotate={(\x-1)*72}]                    
                    \ifnum\x=1 %
                    \draw[thick,->] (0.3,0) -- (2,0) node[midway,above]{$a_\x$};
                    \else
                    \draw[thick,->] (0.3,0) -- (2,0);
                    \fi
                    \node at (0,0) {0};
                    \fill[black] (2.15,0) circle (0.07);
                    \node at (2.5,0) {\x};
                \end{scope}
            }
    \end{tikzpicture}   
    \caption{The star-shaped quiver $Q(5)$.} 
\end{figure}
\end{definition}

There is a short exact sequence
\[ 1 \longrightarrow \SL(2) \longrightarrow G \longrightarrow    T\cong (\CC^*)^n \longrightarrow 1,\]
with the homomorphism $G\rar T$ given by $[(g_0,g_1,\ldots,g_n)] \mapsto (g_1,\ldots,g_n)$. This induces isomorphisms $G^\vee\cong T^\vee\cong \ZZ^n$, which we fix throughout this paper. Let $\Theta := G^\vee\otimes_\ZZ \QQ\cong T^\vee\otimes_\ZZ\QQ\cong \QQ^n$. We write $e_1,\ldots,e_n$ for the standard basis of both $\ZZ^n$ and $\QQ^n$.
\begin{definition}\label{def:polygon}
    Let $\calY:=\Rep(Q,v)$. A polygon space is a GIT quotient of the form \[\calP_\theta := \calY\git_\theta G,\] for some $\theta\in \Theta$.
\end{definition}
Let $V$ denote the natural representation of $\GL(2)$, and let $V^*$ denote its dual representation. Given $W$ a vector space, and $\lambda\in \Theta$, denote by $W(\lambda)$ the representation of $T$ acting with uniform weight $-\lambda$. As a representation of $\GL(Q,v)$, we have
\[ \calY \cong \bigoplus_{i=1}^n V^*(e_i).\]
Fixing bases for $V_i$, for all $i\in Q_0$, we obtain induced bases $x_i,y_i$ of $\Hom(V_0,V_i)^*$. We obtain a presentation 
\[\CC[\calY]\cong \CC[x_1,y_1,\ldots,x_n,y_n],\] with $\deg(x_i)=\deg(y_i)=e_i$.
\begin{definition}\label{def:gr}
    Let $\CC[\mathrm{Gr}(2,n)]$ denote the homogeneous coordinate ring of the Grassmannian of planes in affine $n$-space. This is generated by Plücker coordinates $\varphi_{i,j}$, for $1\leq i<j\leq n$. We endow this with a $\ZZ^n$-grading such that $\deg(\varphi_{i,j})=e_i+e_j$.
\end{definition}
\begin{lemma}\label{lem:griso}
    Let $Y:=\calY\git \SL(2)$. There is an isomorphism of $\ZZ^n$-graded $\CC$-algebras $\CC[Y]\cong \CC[\mathrm{Gr}(2,n)]$. Moreover given $\theta\in \Theta$, we have that $\calY\git_\theta G\cong Y\git_\theta T$.
\end{lemma}
\begin{proof}
    We have that $\CC[Y]=\CC[\calY]^{\SL(2)}$. By the first fundamental theorem of invariant theory for $\SL(2)$ \cite[11.1.2]{Procesi}, we find that $\CC[Y]$ is generated by the determinants $\varphi_{i,j}:=x_iy_j-x_jy_i$, for $1\leq i<j\leq n$. The second statement follows from the fact that $\SL(2)$ admits no non-trivial characters: we have that $\CC[\calY]_\theta \subseteq \CC[\calY]^{\SL(2)}$, hence $\CC[\calY]_\theta = \CC[Y]_\theta$. This is an application of \cite[Lemma~2.1]{alinew}.
\end{proof}
\begin{proposition}\label{lem:polycox}\label{prop:cox}
    Let $n\geq 5$. There is a $T$-equivariant isomorphism 
    \[Y\cong \Spec(\Cox(\PP^{n-3}_{n-1})).\]
\end{proposition}
\begin{proof}
    Consider the action of $\mathbb{G}_a$ on $A:=\CC[x_1,\ldots,x_{n-1},y_1,\ldots,y_{n-1}]$, where $t$ acts via
    \[ x_i\mapsto x_i, \ \ y_i \mapsto y_i + tx_i.\]
    By a theorem of Mukai \cite{Mukai}, the Cox ring of $\PP^{n-3}_{n-1}$ can be calculated as the invariant subring $A^{\mathbb{G}_a}$. The invariant subring is graded by $\ZZ^n$ with basis $e_0,\ldots,e_{n-1}$. An explicit calculation is given in \cite{ottem}, showing that $A^{\mathbb{G}_a}$ is generated by $x_i$, and $x_iy_j-x_jy_i$, and that $A^{\mathbb{G}_a}$ is isomorphic to the homogeneous coordinate ring of $\mathrm{Gr}(2,n)$. We have that $\deg(x_i)=e_i$ and $\deg(x_iy_j-x_jy_i)=e_0-\sum_{k\neq i,j}e_k$. Considering the degrees of elements under the new basis $f_0,\ldots,f_{n-1}$, where $e_0 = \sum_{i=1}^{n-1}f_i$, and $e_i = f_0 + f_i$, we find that the $\ZZ^n$-grading on $A^{\mathbb{G}_a}$ agrees with the grading given in Definition \ref{def:gr}. Hence the equivariant isomorphism $Y\cong\Spec(\Cox(\PP^{n-3}_{n-1}))$ is induced by the isomorphism of graded algebras given in Lemma \ref{lem:griso}
\end{proof}
We briefly describe the GIT quasi-fan associated to the action of $G$ on $\Rep(Q,v)$. For details see \cite[Example~3.3]{dolgachevhu}, or \cite[Chapter~11]{Dolgachev}. We then provide some results elucidating the geometry of the wall crossings relating the various polygon spaces.
\begin{definition}\label{def:hyperplanes}
    Define the convex cone \[C_0 := \Big\{(\theta_i)\in \Theta \ \vert \ 0 \leq \theta_i \leq \sum_{j\neq i}\theta_j, \ \text{for all }i\in [n]\Big\}\subseteq \Theta.\] Given $I\subseteq [n]$, define the hyperplane
\[ H_I:=\Big\{(\theta_i)\in \Theta \ \vert \ \sum_{i\in I}\theta_i=\sum_{j\notin I}\theta_j\Big\}.\] 
    Let $\mathcal{A}$ be the collection of hyperplanes in $\Theta$ given by $H_I$, for all $I\subseteq[n]$, together with the coordinate hyperplanes. 
\end{definition}
\begin{lemma}[\cite{Dolgachev}, Corollary 11.3]\label{lem:polygit}
    We have that $\calY^{\theta\text{-ss}}\neq\emptyset$ if and only if $\theta\in C_0$. Each GIT chamber is given by the set of rational points of a connected component of 
    \[\RR_{>0}\cdot(C_0\cap (\Theta\setminus\bigcup_{H\in\mathcal{A}}H))\subseteq \Theta\otimes_\QQ\RR.\]
\end{lemma}
\begin{corollary}\label{cor:morichamber}
    The number of Mori chambers of $\PP^{n-3}_{n-1}$ is given by the number of chambers of the hyperplane arrangement $\mathcal{A}$ cointained in $C_0$. 
\end{corollary}
Let $M(n)$ denote the number of chambers of the hyperplane arrangement $\mathcal{A}$ contained in $C_0$. Let $C(n)$ denote the number of chambers of $\mathcal{A}$ contained in the positive orthant, then $M(n)=C(n)-n$. The numbers $M(n)$ are given for $n\leq 9$ in \cite[Corollary~1.4]{BCSRW}.
\medskip
\begin{center}
        \begin{tabular}{l|ccccccl}\toprule
            \ \ \ \ \ $n$ & 5 & 6 & 7 & 8 & 9 \\
            $M(n)$ & 76 & 1678 & 122914 & 33207248 & 34448225380 \\ \bottomrule
        \end{tabular}
    \end{center}
    \medskip
\begin{lemma}\label{lem:extquot}
    Let $i\in [n]$. The set 
    \[\theta\in \{\theta\in\Theta \ \vert \ \theta_i+\theta_j > \sum_{k\neq i,j}\theta_k, \ \text{for all }j\neq i\}\cap C_0^\circ\] 
   is a GIT chamber. For $\theta$ in this chamber we have $\calP_\theta\cong\PP^{n-3}$.
\end{lemma}
\begin{proof}
    An analysis of $\theta$-semistability via King stability (see \cite{King}) shows that the complement of $\calY^{\theta\text{-ss}}$ in $\calY$ is covered by the closed subsets defined by the ideals
    \begin{enumerate}
        \item $(\varphi_{i,j})\subseteq\CC[\calY]$, for $j\neq i$.
        \item $(\varphi_{j_1,j_2} \ \vert \ \text{for all }j_1,j_2\neq i)\subseteq\CC[\calY]$. 
    \end{enumerate}
    Let $\calS=\{p\in \calY^{\theta\text{-ss}} \ \vert \ (x_i(p),y_i(p))=(1,0), \ y_j(p)=1\}$. For all $y\in \calY^{\theta\text{-ss}}$ there exists some $g\in G$ such that $g\cdot y \in \calS$. Let $H\subseteq G$ be the subgroup preserving $\calS$, then $H\cong\mathrm{Af}_1$ the group of affine-linear transformations on $\Af^1$, generated by $[(g_0,\ldots,g_n)]$ where $g_0=\begin{pmatrix}
        a & b \\ 0 & 1
    \end{pmatrix}$, and $g_j=1$, for all $j\neq 0$.

    We have $\CC[\overline{\calS}]\cong\CC[x_j \ \vert \ j\neq i]$, and $\CC[\overline{\calS}]^{\mathbb{G}_a}\cong \CC[x_{j_1}-x_{j_2} \ \vert \ \text{for all }j_1,j_2\neq i]\cong \CC[z_0,\ldots,z_{n-3}]$, where $\mathbb{G}_a\subseteq\mathrm{Af}_1$ is the unipotent radical. The residual torus $H/\mathbb{G}_a\cong\CC^*$ acts with weight 1 on each $z_j$. Now note that $\varphi_{j_1,j_2}\vert_\calS = x_{j_1}-x_{j_2}$, hence the image of the ideal $(\varphi_{j_1,j_2} \ \vert \ \text{for all }j_1,j_2\neq i)$ in $\calS$ equals $(z_0,\ldots,z_{n-3})$. Therefore $\calY^{\theta\text{-ss}}/T\cong \PP^{n-3}$. 
\end{proof}
\begin{definition}
    Let $X$ be a smooth variety containing a smooth $\PP^m$, with $m>0$, of codimension $k$, with normal bundle $\oh_{\PP^m}(-1)^{\oplus k}$. We obtain a new variety $X'$ by blowing up $\PP^m\subseteq X$. The exceptional divisor is isomorphic to $\PP^m\times\PP^{k-1}$. There is then a contraction $X'\rar X''$, given by the projection $\PP^m\times \PP^{k-1}$ onto the second factor. The variety $X''$ then contains a smooth $\PP^{k-1}$ of codimension $m+1$, with normal bundle $\oh_{\PP^{k-1}}(-1)^{\oplus m+1}$. We say $X$ and $X''$ are related by a standard flip/antiflip.
\end{definition}
\begin{lemma}\label{lem:flipflop}
    Let $\theta_+\in \Theta$ lie in a GIT chamber $\gamma$, such that a hyperplane $H_I$ supports a facet of $\overline{\gamma}$, with $\# I \geq 2$. Let $\theta_-$ lie in the GIT chamber adjacent to $\gamma$ given by crossing $H_I$. Then $\calP_{\theta_+}$ and $\calP_{\theta_-}$ are related by a standard flip/antiflip if $\#I \geq 3$, and a blow-up of a single point if $\#I=2$.
\end{lemma}
\begin{proof}
    Let $m=n-\#I -1$. We write $\calY_m$, $T_m$ and $\Theta_m$, for $\calY$, $T$, and $\Theta$, but with $m$ replacing $n$ in the definition. Assume that $I=\{m,m+1,\ldots,n\}$. Let 
    \[\theta' = (\theta_1,\ldots,\theta_{m-1}, \sum_{i\in I}\theta_i)\in \Theta_m\] 
    Then $\theta'$ lies in the GIT chamber in $\Theta_m$ described in Lemma \ref{lem:extquot}. Hence $\calY_m\git_{\theta'}T_m \cong \PP^{m-3}$.

    Let $Z\subseteq \calY$ be the closed subscheme defined by the ideal $(\varphi_{i,j} \ \vert \ i,j\in I)$. By analysing the criteria for King semi-stability we see that the locus in $\calY^{\theta\text{-ss}}$ which becomes unstable when we cross the wall $H_I$ is $Z^{\theta\text{-ss}}:= Z_I\cap \calY^{\theta\text{-ss}}$. The morphism $Z^{\theta\text{-ss}}\rar \calY_m^{\theta'\text{-ss}}$ given by projecting onto the first $m$-coordinates is a torus bundle, hence $Z^{\theta\text{-ss}}/T\cong \PP^{m-3}$. From the proof of Lemma \ref{lem:extquot}, we see that the invertible sheaves generated by $\varphi_{i,j}$, such that $i\neq j$ and $\{i,j\}\not\subseteq I$ descend to $\oh_{\PP^m}(1)$. The Plücker relations can then be used to show that the $T$-equivariant invertible sheaves generated by $\varphi_{i,j}$ descend to $\oh_{\PP^m}(1)$, for all $i\neq j$. Thus the conormal bundle descends to $\oh_{\PP^m}(1)^{\oplus n-m}$.

    An identical analysis on the opposite side of the wall $H_I$ shows that the birational transformations $\calP_{\theta_+}\dashrightarrow \calP_{\theta_-}$ are exactly those described in the lemma.
\end{proof}
\subsection{Hyperpolygon spaces.}
We now summarise some results on hyperpolygon spaces, the algebraic symplectic analogue of polygon spaces. Hyperpolygon spaces are the Nakajima quiver varieties associated to the doubled star-shaped quiver. For generalities on the construction of Nakajima quiver varieties see \cite{Nakajima}. Fix $n\geq 4$.
\begin{definition}
    The doubled star-shaped quiver $\overline{Q}$ is the quiver with vertex set $Q_0=\{0,1,\ldots,n\}$ and arrow set $Q_1=\{a_1,b_1,\ldots,a_n,b_n\}$ with $s(a_i)=t(b_i)=0$ and $t(a_i)=s(b_i)=i$, for $1\leq i \leq n$.
\end{definition}
\begin{figure}[ht]
    \centering
    \begin{tikzpicture}[scale=0.75]
        \foreach \x in {1,2,3,4,5} {
            \begin{scope}[rotate={(\x-1)*72}]
                \ifnum\x=1 %
                \draw[thick,->] (0.2,-0.07) to[bend right=15] (2,-0.07) node[below=0.35cm,left=0.3cm] {$a_\x$};
                \draw[thick,<-] (0.2,0.07) to[bend left=15] (2,0.07) node[above=0.35cm,left=0.3cm] {$b_\x$};
                \else
                \draw[thick,->] (0.2,-0.07) to[bend right=15] (2,-0.07);
                \draw[thick,<-] (0.2,0.07) to[bend left=15] (2,0.07);
                \fi
                
                \fill[black] (0,0) circle (0.07);
                \fill[black] (2.15,0) circle (0.07);
                \node at (2.5,0) {\x};
            \end{scope}
        }
    \end{tikzpicture}  
    \caption{The doubled star-shaped quiver $\overline{Q}(5)$.}  
\end{figure}
The action of $G:=G(Q,v)=G(\overline{Q},v)$ on $\Rep(\overline{Q},v)$ is a hamiltonian group action. As a representation of $G$ we have
\[ \Rep(\overline{Q},v)\cong \bigoplus_{i=1}^n V^*(e_i)\oplus V(-e_i).\]
There is an associated ($G$-equivariant) moment map 
\[\mu: \Rep(\overline{Q},v)\longrightarrow \bigoplus_{i\in Q_0}\mathrm{End}(V_i).\] Write $\alpha_i= f_{a_i}$ and $\beta_i=f_{b_i}$. We may then describe $\mu$ as follows
\[ \mu(\alpha_1,\beta_1,\ldots,\alpha_n,\beta_n)=\Big(\sum_{i=1}^n \beta_i\circ\alpha_i,\alpha_1\circ\beta_1,\ldots,\alpha_n\circ\beta_n\Big).\]
Fixing a basis for $V_i$ for all $i\in Q_0$, we obtain an induced basis $x_i,y_i$ of $\Hom(V_0,V_i)^*$ with dual basis $z_i,w_i$ of $\Hom(V_0,V_i)$.
We obtain the presentation \begin{displaymath}{\CC[\Rep(\overline{Q},v)]\cong \CC[x_1,y_1,z_1,w_1,\ldots,x_n,y_n,z_n,w_n]=\CC[\calY][z_1,w_1,\ldots,z_n,w_n]}.\end{displaymath}
The ideal in $\CC[\Rep(\overline{Q}.v)]$ defining $\mu^{-1}(0)$ is generated by the elements
\begin{align*}
    \sum_{i=1}^n y_iz_i, \ \sum_{i=1}^n x_iw_i, \ \sum_{i=1}^n x_iz_i, \ \sum_{i=1}^ny_iw_i, \ \text{and }
    x_iz_i+y_iw_i, \ \text{for }i=1,\ldots,n.
\end{align*}

Let $\Theta := G^\vee\otimes_\ZZ \QQ\cong T^\vee\otimes_\ZZ\QQ\cong \QQ^n$.
\begin{definition}\label{def:hyperpolygon}
    A hyperpolygon space is a GIT quotient of the form 
    \[\MM_\theta:=\mu^{-1}(0)\git_\theta G,\] for some $\theta\in \Theta$.
\end{definition}
\begin{definition}\label{def:consymsin}
    A symplectic variety is a normal variety $X$ equipped with a symplectic form $\omega$ on the smooth locus $X_{\mathrm{reg}}$, such that for any resolution the pullback of $\omega$ extends to a regular 2-form on the resolution. A conical symplectic singularity is an affine symplectic variety $X$ such that the coordinate ring is positively graded: $\CC[X]\cong \oplus_{n\geq 0}\CC[X]_n$, and such that $\CC[X]_0=\CC$ and $\omega$ is homogeneous with respect to the induced $\CC^*$-action.
\end{definition}
\begin{proposition}[\cite{BellamySchedler}, Theorem 1.2]
    The affine variety $\MM_0$ is a conical symplectic singularity, and for all $\theta\in \Theta$, the natural morphism $\pi_\theta:\MM_\theta \rar \MM_0$ is a projective crepant partial resolution.
\end{proposition}
\begin{theorem}[\cite{Namikawa}, Main Theorem]
    Let $\pi:X\rar Y$ be a projective crepant resolution of a conical symplectic singularity. Then $X$ is a Mori dream space.
\end{theorem}
\begin{lemma}[\cite{Kaledin}, Lemma 2.11]\label{lem:semismall}
    Let $\pi_1:X_1\rar Y$ and $\pi_2:X_2\rar Y$ be projective crepant resolutions of a conical symplectic singularity. Then $X_1$ and $X_2$ are isomorphic in codimension 1.
\end{lemma}
\begin{example}\label{ex:n=5}
    Let $n=4$. For all $\theta\in\Theta$, $\MM_\theta$ is 2-dimensional and the entire picture is clear: we recover Kronheimer's construction of all the crepant partial resolutions of the $D_4$ surface singularity (see \cite{Kronheimer}). In particular $\MM_0\cong \CC^2/\Gamma$, where $\Gamma\subseteq \SL(2)$ is the binary dihedral group of order 8.

    The GIT quasi-fan is the fan associated to the $D_4$ Weyl hyperplane arrangement. Any two GIT chambers may be identified by suitable reflections, and the quotients obtained are invariant under such reflections. Let $C$ be a GIT chamber. The closure $\overline{C}$ is the cone over 3-simplex. Given $\theta\in C$, then $\MM_\theta\cong \widetilde{\CC^2/\Gamma}$, the minimal resolution. For $\theta_0$ in a face of $\overline{C}$, then $\MM_{\theta_0}\rar \MM_0$ is a crepant partial resolution.
\end{example}
From now on we assume that $n\geq 5$.
\begin{example}
    Let $n=5$. Then $\MM_0\cong \CC^4/\Gamma$, where $\Gamma=Q_8\times_{\ZZ/2\ZZ}D_8$, the quaternion group and the dihedral group of order 8 identified along their centers.
    Donten-Bury and Wi\'{s}niewski study the crepant resolutions of this quotient singularity in \cite{DBW}. They confirm the calculation of Bellamy \cite{Bellamy} by constructing 81 projective crepant resolutions and describing the birational modifications relating them.
\end{example}
\begin{remark}
    In \cite{BCSRW} it is shown that for $n\geq 6$, the affine hyperpolygon space $\MM_0$ is not isomorphic to a quotient singularity $\CC^{2n}/\Gamma$, where $G\subseteq \mathrm{Sp}(2n)$.
\end{remark}
By the same reasoning as in the polygon space case, we may construct an intermediate quotient by $\SL(2)$. For all $\theta\in \Theta$ we obtain
\[\MM_\theta\cong (\mu^{-1}(0)\git\SL(2))\git_\theta T.\] 
The GIT quasi-fan associated to the action of $T$ on $\mu^{-1}(0)\git \SL(2)$ is described explicitly in \cite{BCSRW}. When describing the GIT quasi-fan associated to the action of $T$ on $\mu^{-1}(0)\git \SL(2)$, given $I\subseteq[n]$, we introduced a collection of hyperplanes $\mathcal{A}$ (see Definition \ref{def:hyperplanes}).
\begin{definition}\label{def:F}
    We define a quasi-fan as follows: a chamber is given by the rational points of a single connected component of $\RR_{>0}\cdot(\Theta\setminus \bigcup_{H\in\mathcal{A}}H)\subseteq\Theta\otimes_\QQ\RR$. The GIT cones are then the closures of these chambers, together with all their faces. Let \[F=\{(\theta_i)\in \Theta \ \vert \ \theta_i\geq 0, \ \text{for all }i\in [n]\}\subseteq \Theta\] denote the positive orthant under the fixed isomorphism $\Theta\cong\QQ^n$ (see discussion above Definition \ref{def:polygon}).
\end{definition}

\begin{theorem}[\cite{BCSRW}, Theorem 3.4]\label{thm:BCSRW}
    The quasi-fan determined by the hyperplane arrangement $\mathcal{A}$ is the GIT quasi-fan, and $\pi_\theta:\mathfrak{M}_\theta\rar\MM_0$ is a crepant resolution if and only if $\theta$ is generic. In other words $\MM_\theta$ is smooth if and only if $\theta\in \Theta\setminus \bigcup_{H\in\mathcal{A}}H$. Moreover any projective crepant resolution $\pi:\MM\rar \MM_0$ is isomorphic to $\pi_\theta:\MM_\theta\rar \MM_0$ (over $\MM_0$), for $\theta$ in a unique chamber in $F$.
\end{theorem}

\begin{corollary}
    The number of projective crepant resolutions $\pi:\MM\rar \MM_0$ up to isomorphism over $\MM_0$, is given by the number of the chambers in the hyperplane arrangement $\mathcal{A}$ contained in $F$. 
\end{corollary}
Let $P(n)$ be the number of chambers in $F$, for $n\leq 9$ these are given by
\medskip
    \begin{center}
        \begin{tabular}{l|ccccccl}\toprule
            \ \ $n$ & 5 & 6 & 7 & 8 & 9 \\
            $P(n)$ & 81 & 1684 & 122921 & 33207256 & 34448225389 \\ \bottomrule
        \end{tabular}
    \end{center}
\section{Cox rings}
\subsection{The Cox ring of smooth hyperpolygon spaces.}
Let $V$ denote the natural representation of $\GL(2)$, and $V^*$ its dual representation. Fix bases $x_i,y_i$ of $V(-e_i)$ and $c_i$ of $\det(V)^*(2e_i)$, for $i=1,\ldots,n$. Let 
\[\calW=\bigoplus_{i=1}^n \big(V^*(e_i)\oplus \det(V)(-2e_i)\big).\] We obtain a presentation
\[ \CC[\calW]\cong \CC[x_1,y_1,c_1,\ldots,x_n,y_n,c_n].\]
Let $J$ be the $\GL(Q,v)$-invariant ideal generated by
\[ \sum_{i=1}^n c_ix_i^2, \ \ \sum_{i=1}^n c_ix_iy_i, \ \ \sum_{i=1}^n c_iy_i^2.\]
\begin{definition}\label{def:calX}\label{def:RI}
    Let $\calX = \Spec(\CC[\calW]/J)$. The subgroup $\Delta\subseteq \GL(Q,v)$ acts trivially on $\calX$, thus $\calX$ has a $G$-action. Let
    \[ R = \CC[\varphi_{i,j},c_k \ \vert \ 1\leq i<j\leq n, \ k=1,\ldots,n].\]
    Endow $R$ with a $\ZZ^n$-grading via $\deg(\varphi_{i,j})=e_i+e_j$, and $\deg(c_k)=-2e_k$. Let $I\subseteq R$ be the homogeneous ideal generated by the following sets of elements
    \begin{enumerate}
        \item $\varphi_{i,j}\varphi_{k,l}-\varphi_{i,k}\varphi_{j,l}+\varphi_{i,l}\varphi_{j,k}$, for $1\leq i<j<k<l\leq n$.
        \item $\sum_{k=1}^n\varphi_{i,k}\varphi_{j,k}c_k$, for $1\leq i\leq j \leq n$.
    \end{enumerate}
\end{definition}
The remainder of this section is dedicated to the proof of the following theorem.
\begin{theorem}\label{thm:cox}
    Fix $n\geq 5$. The following $\ZZ^n$-graded $\CC$-algebras are isomorphic.
    \begin{enumerate}
        \item $\Cox(\mathfrak{M}_\theta)$, for $\theta\in\Theta$ generic.
        \item $\CC[\calX]^{\SL(2)}$.
    \item $R/I$.
    \end{enumerate}
\end{theorem} 
\begin{definition}
    For $\theta\in \Theta$ generic, the \emph{core} of $\MM_\theta$ is the zero fiber $\pi_\theta^{-1}(0)$. In \cite[Theorem~5.8]{NakALE} it is shown that every irreducible component of $\pi_\theta^{-1}(0)$ is a Lagrangian subvariety. In particular each irreducible component of the core has dimension $n-3$.
\end{definition}
\begin{definition}
    Fix $\theta \in \Theta$. We say $I\subseteq [n]$ is \emph{$\theta$-long} if
    \[ \sum_{i\in I}\theta_i>\sum_{j\notin I}\theta_j.\]
    and $I$ is \emph{$\theta$-short} if
    \[ \sum_{i\in I}\theta_i<\sum_{j\notin I}\theta_j.\]
    If $\theta\in \Theta$ is generic then every subset $I\subseteq[n]$ is either $\theta$-short or $\theta$-long.
    For every $I\subseteq [n]$, define an ideal 
    \[ \mathscr{I}_I = (x_iy_j-x_jy_i, z_k,w_k \ \vert \ i,j\in I, \ k\in I^c)\subseteq\CC[\mu^{-1}(0)].\]
\end{definition}
\begin{theorem}[\cite{Konno}, Theorem 4.2]\label{thm:konno}
    Let $\theta\in F$ be generic. Then 
    \[ \mu^{-1}(0)^{\theta\text{-ss}}= \mu^{-1}(0)\setminus\left(\bigcup_{ \theta\text{-long}\ I}\VV(\mathscr{I}_I)\cup\bigcup_{i=1}^n \VV(x_i,y_i)\right).\]
\end{theorem}
\begin{theorem}[\cite{HP}, Theorem 2.2]\label{thm:HPcore}
    Let $\theta\in F$ be generic. Let $\Pi_\theta:\mu^{-1}(0)^{\theta\text{-ss}}\rar \MM_\theta$ denote the quotient morphism. Then
    \[ \pi_\theta^{-1}(0)= \Pi_\theta\left(\mu^{-1}(0)^{\theta\text{-ss}}\cap \bigcup_{I}\VV(\mathscr{I}_I)\right).\]
\end{theorem}
\begin{lemma}\label{lem:codimen}
    Let $\theta_1,\theta_2\in F$ be generic. Then 
    \[ \MM_{\theta_1}\setminus\pi_{\theta_1}^{-1}(0) \cong \MM_{\theta_2}\setminus\pi_{\theta_2}^{-1}(0).\]
    Moreover $\mu^{-1}(0)^{\theta_1\text{-ss}}$ and $\mu^{-1}(0)^{\theta_2\text{-ss}}$ are isomorphic in codimension 1.
\end{lemma}
\begin{proof}
    For $\theta\in F$ generic, let
    $N_\theta:= (\Pi_\theta\circ\pi_\theta)^{-1}(0)$. This is the preimage of the core in $\mu^{-1}(0)$. By Theorem \ref{thm:konno} together with Theorem \ref{thm:HPcore}, we see that $\mu^{-1}(0)^{\theta_1\text{-ss}}\setminus N_{\theta_1}=\mu^{-1}(0)^{\theta_2\text{-ss}}\setminus N_{\theta_2}$. Therefore we obtain that $\MM_{\theta_1}$ and $\MM_{\theta_2}$ are isomorphic away from their cores.
    
    Since $\Pi_\theta$ is a geometric quotient by $G$, where $\dim(G)=n+3$, and $\pi_\theta^{-1}(0)$ has dimension $n-3$, we see that $\dim(N_\theta)=2n$. By \cite[Theorem~1.2]{CB}, we have that $\mu^{-1}(0)$ is an irreducible complete intersection of dimension $3n-3$. Therefore $N_\theta$ is a closed subscheme of $\mu^{-1}(0)^{\theta\text{-ss}}$ of codimension $n-3$. Since we have assumed $n\geq 5$, we see that $N_\theta$ has codimension $\geq 2$ in $\mu^{-1}(0)^{\theta\text{-ss}}$. Hence $\mu^{-1}(0)^{\theta_1\text{-ss}}$ and $\mu^{-1}(0)^{\theta_2\text{-ss}}$ contain the same codimension 1 points.
\end{proof}
\begin{proposition}\label{prop:calXmap}
    There is a $G$-equivariant morphism $\iota:\calX \rar \mu^{-1}(0)$ such that for all generic $\theta\in F$ we have that $\iota\vert_{\iota^{-1}(\mu^{-1}(0)^{\theta\text{-ss}})}$ is an isomorphism. Moreover $\calX$ is normal and $\iota^{-1}(\mu^{-1}(0)^{\theta\text{-ss}})$ has complement of codimension $\geq 2$, so restriction of functions induces an isomorphism
    \[ \CC[\calX]\cong \Gamma(\mu^{-1}(0)^{\theta\text{-ss}},\oh_{\mu^{-1}(0)^{\theta\text{-ss}}}).\]
\end{proposition}
\begin{proof}
    We begin by constructing the morphism. A presentation of $\CC[\mu^{-1}(0)]$ is given above Definition \ref{def:hyperpolygon}. Consider the algebra map $\iota^*:\CC[\Rep(\overline{Q},v)]\rar \CC[\calX]$ defined on generators by
    \[ x_i\mapsto x_i, \ \ y_i \mapsto y_i, \ \ z_i \mapsto c_iy_i, \ \ w_i \mapsto -c_ix_i.\]
    The ideal defining $\mu^{-1}(0)$ in $\CC[x_1,y_1,z_1,w_1,\ldots,x_n,y_n,z_n,w_m]$ is generated by
    \begin{align*}
        \sum_{i=1}^n y_iz_i, \ \sum_{i=1}^n x_iw_i, \ \sum_{i=1}^n x_iz_i, \ \sum_{i=1}^ny_iw_i, \ \text{and }
        x_iz_i+y_iw_i, \ \text{for }i=1,\ldots,n.
    \end{align*}
    We have $f(\sum_{i=1}^n y_iz_i)=\sum_{i=1}^n c_iy_i^2\in J$, and $f(x_iz_i+y_iw_i)=x_ic_iy_i-y_ic_ix_i=0$, for all $i\in[n]$. Similarly all generators of the ideal can be seen to have image inside $J$, hence $\iota^*$ factors through $\CC[\mu^{-1}(0)]$. We have that $c_i\in \det(V^*)(2e_i)$, $x_i,y_i\in V(-e_i)$ and $z_i,y_i\in V^*(e_i)$. Thus $c_iy_i,-c_ix_i\in \det(V^*)\otimes V(e_i)\cong V^*(e_i)$, and from this one deduces that $\iota^*$ respects the $G$-action.

    Let $\iota:\calX \rar \mu^{-1}(0)$ denote the $G$-equivariant morphism induced by $\iota^*$.
    Let 
    \[U:=\{p\in\mu^{-1}(0) \ \vert \ (p_i,q_i)\neq 0, \ \text{for all }i=1,\ldots n\}\subseteq\mu^{-1}(0).\]
    On $U$ there is an inverse morphism given on functions by
    \[ x_i\mapsto x_i, \ \ y_i \mapsto y_i, \ \ c_i \mapsto \frac{z_i}{y_i}=-\frac{w_i}{x_i}.\]
    Thus $\iota\vert_{\iota^{-1}(U)}$ is an isomorphism.
    Let $i\in [n]$. Suppose $\theta \in F$ such that $\theta_i>0$, then $(\CC[\mu^{-1}(0)]/(x_i,y_i))_{\theta}=0$, since any homogeneous function such that the $i$-th component of its degree is positive must lie in the ideal $(x_i,y_i)$. Therefore for all $\theta\in F^\circ$, we have that $\mu^{-1}(0)^{\theta\text{-ss}}\subseteq U$.

     We claim $\calX$ is normal. Let $\calX_{n-1}$ denote the variety given by the same presentation as $\calX$, but with $n-1$ replacing $n$. Then $\dim(\calX_{n-1})< 3n-3$, since it is a closed subscheme defined by a non-zero ideal in $\CC^{3n-3}$. Now note that $\CC[\calX]/(x_i,y_i)\cong \CC[\calX_{n-1}\times \CC]$. Therefore the closed subschemes in $\calX$ defined by the ideals $(x_i,y_i)$ have dimension $< 3n-2$. Since $\calX$ contains a dense open subset $\iota^{-1}(U)$ of dimension $3n-3$, and the complement of $U$ has dimension $<3n-3$, $\calX$ is an irreducible complete intersection. By the unmixedness theorem \cite[Theorem~17.6]{Matsumura}, $\calX$ is reduced. By the same reasoning $\dim(\calX_{n-1})=3n-6$, and the complement of $\iota^{-1}(U)$ has codimension $\geq 2$. Therefore $\iota$ induces an isomorphism on the local rings of all codimension 1 points in $\calX$. Since $\mu^{-1}(0)$ is normal, it is regular in codimension 1, therefore $\calX$ is regular in codimension 1. Since $\calX$ is a complete intersection this suffices to show $\calX$ is normal (see II.8.23 in \cite{Hartshorne}).

    Define
    \[ \mu^{-1}(0)^{F\text{-ss}}=\bigcup_{\substack{\theta\in F \\ \mathrm{generic}}}\mu^{-1}(0)^{\theta\text{-ss}}.\]
    It follows from Lemma \ref{lem:codimen} that the inclusion $\mu^{-1}(0)^{\theta\text{-ss}}\subseteq \mu^{-1}(0)^{F\text{-ss}}$ is an isomorphism in codimension 1. Moreover it follows from Theorem \ref{thm:konno}, that $\mu^{-1}(0)^{F\text{-ss}}\subseteq U$.

    We claim that $\iota^{-1}(\mu^{-1}(0)^{\theta\text{-ss}})$ has complement of codimension $\geq 2$ in $\calX$. If this holds, then since $\calX$ is normal, we obtain the isomorphism given by restriction of functions
    \[ \CC[\calX]\cong \Gamma(\mu^{-1}(0)^{\theta\text{-ss}},\oh_{\mu^{-1}(0)^{\theta\text{-ss}}}),\]
    by Hartogs' lemma (see II.6.3A in \cite{Hartshorne}).

    We know that $\mu^{-1}(0)^{\theta\text{-ss}}\subseteq \mu^{-1}(0)^{F\text{-ss}}$ is an isomorphism in codimension 1, and that $U\cong \iota^{-1}(U)\subseteq \calX$ is an isomorphism in codimension 1. Thus the claim reduces to showing that $\mu^{-1}(0)^{F\text{-ss}}\subseteq U$ is an isomorphism in codimension 1.

    It follows from Theorem \ref{thm:konno} that
    \[ \mu^{-1}(0)^{F\text{-ss}}=\mu^{-1}(0)\setminus \left( \VV(\mathscr{I}_{[n]})\cup\bigcup_{i=1}^n\VV(x_i,y_i)\right) = U\setminus \left(\VV(\mathscr{I}_{[n]})\cap U\right).\]
    We describe $Z:=\iota^{-1}\left(\VV(\mathscr{I}_{[n]})\cap U\right)$. Let 
    \[Z'=\VV(x_iy_j-x_jy_i \ \vert \ i,j\in[n])\subseteq \Spec(\CC[x_1,y_1,c_1,\ldots,x_n,y_n,c_n]).\] 
    Let
    \[ U' = \Spec(\CC[x_1,y_1,c_1,\ldots,x_n,y_n,c_n]) \setminus \bigcup_{i=1}^n\VV(x_i,y_i).\]
    Then note $U'\cap Z'$ is just the product of $\CC^n$ with the space of $n$ non-zero pairwise linearly independent vectors in $\CC^2$, therefore $Z'$ is integral and $\dim(U'\cap Z')= 2n+1$. We have that $Z = \calX \cap U'\cap Z'$. The equations in the ideal $J$ defining $\calX$ give non-zero relations in the coordinate functions of the $\CC^n$ factor (see Definition \ref{def:calX}). Thus $\dim(Z)\leq 2n$. Therefore $Z$ has codimension $\geq 2$ in $\iota^{-1}(U)$, finishing the proof of the proposition.
\end{proof}
\begin{corollary}
    Let $\theta\in \Theta$ be generic. There is an isomorphism of $\ZZ^n$-graded $\CC$-algebras
    \[ \Cox(\MM_\theta)\cong \CC[\calX]^{\SL(2)}.\]
\end{corollary}
\begin{proof}
    It is shown in Theorem 1.2 of \cite{BCSRW} that there is an isomorphism
    \begin{align*}
        L_\theta: G^\vee &\longrightarrow \Pic(\MM_\theta), \\
        \chi &\longmapsto (\pi_{\theta})_*(\oh_{\mu^{-1}(0)^{\theta\text{-ss}}}\otimes \chi)^G.
    \end{align*}
    This is the homomorphism given by equivariant descent. Theorem 3.2 in \cite{alinew} provides an isomorphism 
    \[ \Gamma(\mu^{-1}(0)^{\theta\text{-ss}},\oh_{\mu^{-1}(0)^{\theta\text{-ss}}})^{\SL(2)}\cong \Cox(\MM_\theta).\]
    Thus the claim follows from Proposition \ref{prop:calXmap}.
\end{proof}
Now we compute an explicit presentation of $\CC[\calX]^{\SL(2)}$.
\begin{lemma}
    Let $\varphi_{i,j}:=x_iy_j-x_jy_i\in \CC[\calX]$, for $1\leq i<j\leq n$. The $\SL(2)$-invariant subalgebra $\CC[\calX]^{\SL(2)}$ is generated by $\varphi_{i,j}$, for $1\leq i<j\leq n$ and $c_k$, for $k=1,\ldots,n$.
\end{lemma}
\begin{proof}
    B the first fundamental theorem of invariant theory for $\SL$, we have that $\CC[V^{\oplus n}]^{\SL(2)}$ is generated by $\varphi_{i,j}$, for $1\leq i <j \leq n$. Moreover since $\SL(2)$ acts trivially on $\det(V)$ we have that $c_i$ is $\SL(2)$-invariant, for $i\in [n]$. Since there is an $\SL(2)$-equivariant surjection $\CC[(V\oplus\CC)^{\oplus n}]\rar \CC[\calX]$ the result follows since $\SL(2)$ is reductive, thus taking invariants is right exact.
\end{proof}
\begin{lemma}There is an isomorphism of $\ZZ^n$-graded $\CC$-algebras
    \[ \CC[\calX]^{\SL(2)}\cong R/ I,\]
    where $R$ and $I$ are as given in Definition \ref{def:RI}.
\end{lemma}
\begin{proof}   
    Recall $\CC[\calY]=\CC[\oplus_{i=1}^n V^*(e_i)]$ (see the discussion above Definition \ref{def:gr}). Let $M:=\oplus_{i=1}^n \CC[\calY](-2e_i)$, we have that 
    \[\Sym_{\CC[\calY]}(M)=\CC[\calY][c_1,\ldots,c_n]\cong\CC[\calW].\] Note that $\deg(c_i)=-2e_i$ and $c_i$ is $\SL(2)$-invariant for each $i$. Let 
    \[I_0=(\varphi_{i,j}\varphi_{k,l}-\varphi_{i,k}\varphi_{j,l}+\varphi_{i,l}\varphi_{j,k} \ \vert \ 1\leq i<j<k<l\leq n).\] By the first fundamental theorem of invariant theory we have $\CC[\calW]^{\SL(2)}=R/I_0$.

    As a representation of $\SL(2)$, the vector space spanned by the three equations in $\CC[\calW]$ defining $\Sigma^{-1}(0)$ is isomorphic to $\Sym^2(V^*)$. This induces an $\SL(2)$-equivariant presentation of $\CC[\calY]$-modules
    \[ \CC[\calW]\otimes \Sym^2(V^*) \rar \CC[\calW] \rar \CC[\calX]\rar 0.\]
    Taking $\SL(2)$-invariants we obtain
    \[ (\CC[\calW]\otimes \Sym^2(V^*))^{\SL(2)} \rar R/I_0 \rar \CC[\calX]^{\SL(2)}\rar 0.\]
    Thus it suffices to find the images of the generators of $(\CC[\calW]\otimes \Sym^2(V^*))^{\SL(2)}$ in $R$.
    Consider the $\SL(2)$-equivariant $\CC[\calY]$-module $\CC[\calW]\otimes_\CC \CC[V]$. This module contains $\CC[\calW]\otimes \Sym^2(V^*)$ as the $\CC[\calY]$-submodule consisting of functions of degree 2 on $V^*$.
    
    Let $\CC[V^*]\cong\CC[w_1,w_2]$. By the first fundamental theorem for invariant theory the $\SL(2)$-invariant functions of degree 2 on $V^*$ are given by
    \[ (x_iw_1-y_iw_2)(x_jw_1-y_jw_2), \ \text{for } 1\leq i \leq j \leq n.\]
    The image of such an element in $\CC[\calW]$ under the presentation $\CC[\calW]\rar \CC[\calX]$ is given by
\[ \sum_{k=1}^n (x_ix_jx_k^2 - x_ky_k(x_iy_j+x_jy_i) + y_iy_jy_k^2)c_k = \sum_{k=1}^n\varphi_{i,k}\varphi_{j,k}c_k.\]
    Therefore $\CC[\calX]^{\SL(2)}\cong R/I$.
\end{proof}
\section{Non-projective polygon spaces}
\subsection{The orbit cones of $Y$.}
Recall from Definition \ref{def:polyquiver} that $\calY=~\Rep(Q,v)$ is the vector space of representations of the undoubled star-shaped quiver with dimension vector $v$. By Lemma \ref{lem:griso} the affine quotient $Y=\calY\git\SL(2)$ is isomorphic to the affine cone over the Grassmannian $\mathrm{Gr}(2,n)$ under its Plücker embedding. The homogeneous coordinate ring $\CC[\mathrm{Gr}(2,n)]$ is endowed with a $\ZZ^n$-grading, where the degree of the Plücker coordinate $\varphi_{i,j}$ is $e_i+e_j$. The projective GIT quotients arising from this setup are well-studied and included as introductory examples in many texts on GIT (see \cite[Chapter~3]{GIT} or \cite[Chapter~11]{Dolgachev}). By Proposition \ref{prop:cox} the affine variety $Y$ is isomorphic to the spectrum of the Cox ring of a smooth projective variety, hence $Y$ is normal \cite[Theorem~1.5.1.1]{ADHL}.

The $\ZZ^n$-grading on $\CC[Y]$ induces an action of a torus $T\cong(\CC^*)^n$. Since $Y$ is a normal affine variety there is an open subset $Y^\free$ where the torus acts freely. We begin by classifying the orbit cones in $\Omega_Y$ and those in the subset $\Omega_Y^\free=\{\omega_y\in \Omega_Y \ \vert \ y\in Y^\free\}$.
\begin{definition}
    Let $I\subseteq[n]$, and let $P$ be a partition of $I$. Let 
\[ \omega_P\vcentcolon=\mathrm{Cone}(e_i+e_j \ \vert \ \{i,j\}\not\subseteq J, \ \text{for all }J\in P, \ \text{for }i,j\in I)\subseteq \Theta.\]
Where $\cone(S)=\{\lambda_1v_1 + \ldots +\lambda_kv_k \ \vert \ \lambda_i\geq 0, \ v_i\in S\}$, for $S\subseteq \Theta$.
\end{definition}
\begin{lemma}\label{lem:polycones}
    We have
    \begin{enumerate}
        \itemsep0.2em
        \item $\Omega_Y = \{\omega_{P} \ \vert \ P \ \text{partition of }I\subseteq[n]\}$.
        \item $\Omega_Y^\free=\{\omega_P \ \vert \ P \ \text{partition of }[n], \ \text{with }\#P\geq 3\}$.
    \end{enumerate}
\end{lemma}
\begin{proof}
The orbit cone of $y\in Y$ is generated by the degrees of the homogeneous generators of $\CC[Y]$ which are non-zero at $y$; therefore
\[\omega_y = \cone(\deg(\varphi_{i,j}) \ \vert \ \varphi_{i,j}(y)\neq 0).\] 
The quotient $\pi_Y:\calY\rar Y$ is surjective and induced by an injection of coordinate rings, hence $\varphi_{i,j}(y) \neq0$ if and only if $\varphi_{i,j}(p)\neq 0$ for any $p\in \pi_Y^{-1}(y)$.

We first prove statement (1).
Let $y\in Y$, and let $(p_1,q_1,\ldots,p_n,q_n)\in~\pi_Y^{-1}(y)$. Set $I=\{i\in [n] \ \vert \ (p_i,q_i)\neq(0,0)\}$, and let $P$ be the partition of $I$ given by the equivalence relation $i\sim j$ if $p_iq_j-p_jq_i=0$, for $i,j\in I$. One sees that $\omega_y = \omega_P$, and $\Omega_Y\subseteq \{\omega_P \ \vert \ P \ \text{partition of }I\subseteq[n]\}$. For the opposite inclusion let $I\subseteq[n]$, and let $P$ be a partition of $I$. We construct a point $p=(p_1,q_1,\ldots,p_n,q_n)\in\calY$ such that $\pi_Y(p)$ has the orbit cone $\omega_{P}$. Set $(p_i,q_i)=(0,0)$, for all $i\in I$. Fix a collection $\{w_J \ \vert \ J\in P\}$ of pairwise linearly independent vectors in $\CC^2$ indexed by $P$. Set $(p_j,q_j)=w_J$, for all $j\in J$, for all $J\in P$. This proves (1).

To prove statement (2) we describe an open subset $\calY^\free\subseteq \calY$ with the property that the closed points of $\calY^\free$ are the closed points in $\calY$ with trivial stabiliser under the $G$-action. Let $p=(p_1,q_1,\ldots,p_n,q_n)\in \calY$. Suppose $(p_i,q_i)=(0,0)$, for some $i\in [n]$. The subgroup $\CC^*\subseteq G$, which acts by change of basis only at the vertex $i\in Q$, stabilises $p$. Therefore, for any closed point with trivial stabiliser we must have $(p_i,q_i)\neq (0,0)$, for all $i\in [n]$. Suppose $(p_i,q_i)\neq (0,0)$, for all $i\in [n]$. Suppose the induced partition $P$ has size 1, then there is a subgroup $\mathbb{G}_a\subseteq G$, which acts at the vertex $0\in Q_0$, which fixes the 1-dimensional linear subspace given by $\mathrm{Span}((p_i,q_i) \ \vert \ i\in [n])$, which stabilises $p$. Suppose $P$ has size 2, then there is a maximal torus in $\GL(2)$ with image in $G$ that stabilises $p$. Finally supposing $P$ has size $\geq 3$, then supposing $[g_0,g_1,\ldots,g_n]\in G$ stabilises $p$, we must have that $g_0=\lambda I$ (since $g_0$ must have $\geq 3$ linearly independent eigenvectors), and we must have $g_i=\lambda$. Thus $(g_0,g_1,\ldots,g_n)\in \Delta$, hence $p$ has trivial stabiliser. It follows that the closed points with trivial stabiliser lie in the complement of a closed subset. Denote this complement by $\calY^\free$.

By \cite[Proposition~3.1]{GIT}, we have that the restriction of $\pi_Y$ to $\pi_Y^{-1}(Y\setminus\{0\})$ is a locally trivial $\SL(2)$-bundle over $Y\setminus\{0\}$. Moreover $\calY^\free\subseteq \pi_Y^{-1}(Y\setminus\{0\})$. Since the restricted morphism $\pi_Y^{-1}(Y\setminus\{0\})\rar Y\setminus\{0\}$ is a geometric quotient of the $\SL(2)$-action, it follows that a closed point $p\in \pi_Y^{-1}(Y\setminus\{0\})$ has trivial stabiliser if and only if $\pi_Y(p)$ has trivial stabiliser. Hence $\pi_Y^{-1}(Y^\free)=\calY^\free$. The description of $\calY^\free$ as the locus where $(p_i,q_i)\neq0$, for all $i\in[n]$, and such that the induced partition has size $\geq 3$ implies statement (2).
\end{proof}
\begin{definition}
Let $\omega\subseteq \Theta$ be a convex cone. The dual cone to $\omega$ is 
    \[ \omega^\vee\vcentcolon=\{L\in \Theta^* \ \vert \ L(z)\geq 0, \ \text{for all }z\in \omega\}\subseteq \Theta^*.\]
    Where $\Theta^*$ denotes the vector space dual to $\Theta$. Recall $\Theta$ is endowed with a fixed basis $e_i$, for $i\in[n]$. Let $f_i\in \Theta^*$ be the dual basis to $e_i$. For $I\subseteq [n]$, let 
    \[v_I \vcentcolon= \sum_{j\notin I}f_j - \sum_{i\in I}f_i\in\Theta^*.\]
\end{definition}
\begin{definition}\label{def:etacone}
    Let $I\subseteq[n]$. Let
    \[ \eta_I \vcentcolon= \mathrm{Cone}\big(e_i+e_j \ \vert \ \{i,j\}\not\subseteq I, \ \text{for }i,j\in [n]\big).\]
    Note $\eta_I=\omega_{P_I}$, where $P_I$ denotes the partition of $[n]$ consisting of $I$ together with the set of singleton sets in $I^c$.
\end{definition}
\begin{lemma}\label{lem:twoprop}\label{lem:dualdesc}
    Let $\omega_P,\omega_Q\in\Omega_Y^\free$. Then
    \begin{enumerate}
        \itemsep0.2em
        \item $\omega_P^\vee = \cone(v_I \ \vert \ I\in P) + \cone(f_1,\ldots,f_n)\subseteq \Theta^\vee$.
        Where $+$ denotes the Minkowski sum.
        \item $\omega_P\subseteq \omega_Q$ if and only if $Q$ is a refinement of $P$.
        \item if $\omega_P^\circ\cap\omega_Q^\circ=\emptyset$, then there exists some $I\subseteq [n]$ such that $\omega_P\subseteq \eta_I$ and $\omega_Q\subseteq \eta_{I^c}$.
    \end{enumerate}
\end{lemma}
\begin{proof}
    We first prove statement (1). Let 
    \[\chi \vcentcolon= \cone(v_I \ \vert \ I\in P) + \cone(f_1,\ldots,f_n)\subseteq \Theta^\vee.\] 
    Let $I\in P$, and let $i,j\in [n]$ such that $\{i,j\}\not\subseteq J$, for all $J\in P$. Then 
    \[ v_I(e_i+e_j)=\begin{cases}
        0, & \text{if }i \ \text{or} \ j \ \text{lies in} \ I, \\
        2, & i,j\notin I.
    \end{cases} \]
    Moreover, given $k\in [n]$, then 
    \[ f_k(e_i+e_j)=\begin{cases}
        0, & \text{if }i,j\neq k, \\
        1, & \text{if }i \text{ or }j = k.
    \end{cases}
    \]
    It follows that $\chi\subseteq\omega_P^\vee$, since the cone generators of $\chi$ are non-negative on the cone generators of $\omega_P$. For the opposite inclusion let $z\in \chi^\vee$. Let $I\in P$, then by definition $v_I(z)\geq 0$. We can subtract a non-negative multiple of $e_i+e_j$ from $z$, where $i\in I$ and $j\notin I$, such that the result remains in $\chi^\vee$. We can repeat this process for each $I\in P$, when we reach the final $I$ we have that $\sum_{j\notin I}z_j = 0$. Hence $z$ is a non-negative linear combination of the generators of $\omega$. Thus $\chi^\vee \subseteq \omega$. This completes the proof of statement (1).

    To prove statement (2) assume that $Q$ is a refinement of $P$, i.e.\ for all $I\in Q$, there is some $J\in P$ with $I\subseteq J$. Therefore 
    \[v_I=v_J+2\sum_{i\in J\setminus I}f_i\in \omega_P^\vee.\] 
    Thus $\omega_Q^\vee\subseteq \omega_P^\vee$. For the opposite inclusion suppose $\omega_P\subseteq\omega_Q$. Let $I\in Q$, then by assumption $v_I\in \omega_P^\vee$. Suppose for a contradiction that $I\not\subseteq J$, for all $J\in P$. Then there exists distinct $J_0,J_1\in P$ with $j_0\in I\cap J_0$ and $j_1\in I\cap J_1$. We have
    \[ v_I = \sum_{J\in P}\lambda_J v_J + \sum_{i=1}^n \mu_i f_i, \ \ \lambda_J,\mu_i\geq 0.\]
    Taking the $j_0$-th and $j_1$-th coordinates we obtain the following inequalities.
    \begin{align*}
        -1 \geq -\lambda_{J_0} + \sum_{J\neq J_0}\lambda_J, \ \ 
        -1 \geq -\lambda_{J_1} + \sum_{J\neq J_1}\lambda_J.
    \end{align*}
    Summing gives
    $
        -2 \geq \sum_{J\neq J_0,J_1}\lambda_J$.
    Therefore $\lambda_J=0$, for all $J\neq J_0,J_1$. Considering the $j_0$-th and $j_1$-th coordinates again we have
    \begin{align*}
        -1 &= -\lambda_{J_0}+\lambda_{J_1}+\mu_{j_0},\\
        -1 &= \lambda_{J_0}-\lambda_{J_1}+\mu_{j_1}.
    \end{align*}
    Summing implies that $\mu_{j_0}=\mu_{j_1}=0$, which implies a contradiction. Hence there exists some $J\in P$ with $I\subseteq J$, and $Q$ is a refinement of $P$. This proves statement (2).

 We now prove statement (3). Since $\omega\in \Omega_Y^\free$ is a top dimensional cone we have that $\omega_P^\circ\cap\omega_Q^\circ\neq\emptyset$ if and only if $\omega_P\cap \omega_Q$ is top dimensional. Hence we have $\omega_P^\circ \cap \omega_Q^\circ = \emptyset$ if and only if $\omega_P\cap\omega_Q$ is contained in a proper linear subspace; if and only if $(\omega_P\cap\omega_Q)^\vee=\omega_P^\vee+\omega_Q^\vee$ contains a non-zero linear subspace. 
    
Note that by statement (2) we have that $\omega_P\subseteq \eta_I$ and $\omega_Q\subseteq \eta_{I^c}$ if and only if $P_I$ is a refinement of $P$ and $P_{I^c}$ is a refinement of $Q$. Equivalently $I\subseteq I'$, for some $I'\in P$ and $I^c\subseteq J'$, for some $J'\in Q$. In this case we have $(I')^c \cap (J')^c = \emptyset$. We show that if $(I')^c\cap (J')^c\neq\emptyset$, for all $I'\in P$ and $J'\in Q$, then $\omega_P^\vee+\omega_Q^\vee$ does not contain a non-zero linear subspace.
    
The Minkowski sum $\omega_P^\vee+\omega_Q^\vee$ contains a non-zero linear subspace if and only if there exists some non-trivial linear combination with non-negative coefficients:
\[ \sum_{I\in P}\lambda_I v_I + \sum_{J\in Q}\lambda_Jv_J + \sum_{i=1}^n \mu_if_i =0.\]
Suppose there exists such a linear combination. 
Given any $i\in [n]$, let $I_i$ and $J_i$ denote the unique sets in $P$ and $Q$ respectively, containing $i$. Fix $i\in [n]$, taking the $i$-th coordinate of the above linear combination we obtain the inequality:
\[ \lambda_{I_i}+\lambda_{J_i}\geq \sum_{I\neq I_i}\lambda_I + \sum_{J\neq J_i}\lambda_J.\]
By assumption there exists some $j\in I_i^c \cap J_i^c$. Taking the $j$-th coordinate of the linear combination we obtain:
    \begin{align*}
         \lambda_{I_j}+\lambda_{J_j}&\geq \sum_{I\neq I_j}\lambda_I + \sum_{J\neq J_j}\lambda_J, \\
         &= \lambda_{I_i}+\lambda_{J_i} + \sum_{I\neq I_i,I_j}\lambda_I + \sum_{J
         \neq J_i,J_j}\lambda_J, \\
         &\geq \lambda_{I_j}+\lambda_{J_j} + 2\left(\sum_{I\neq I_i,I_j}\lambda_I + \sum_{J\neq J_i,J_j}\lambda_J\right).
    \end{align*}
    Where we have applied the inequality obtained from taking the $i$-th coordinate.
    Thus we have that:
    \[ \sum_{I\neq I_i,I_j}\lambda_I + \sum_{J\neq J_i,J_j}\lambda_J=0.\]
    Hence $\lambda_I=0$, for all $I\neq I_i,I_j$, and $\lambda_J=0$, for all $J\neq J_i,J_j$. 
    
    Finally, let $k\in (I_i\cup I_j)^c$, which is non-empty since $\omega_P\in\Omega_Y^\free$. We are reduced to two cases. Let $k\in (J_i\cup J_j)^c$, then
    \[ 0 \geq \lambda_{I_i}+\lambda_{J_i}+\lambda_{I_j}+\lambda_{J_j},\]
    hence all the coefficients in the linear combination vanish. In the other case, we may assume without loss of generality that $k\in J_i$, which gives the inequality:
    \begin{align*}
         \lambda_{J_i} &\geq \lambda_{I_i}+\lambda_{I_j}+\lambda_{J_j}, \\
         & \geq 2\lambda_{I_i}+\lambda_{J_i}.
    \end{align*}
    Thus $\lambda_{J_i}=0$. Let $k\in I_j^c\cap J_i^c$. Again there are two cases. If $k\in J_j^c$ then all coefficients vanish. If $k\in J_j$ then taking the $k$-th coordinate gives:
    \[ \lambda_{J_j}\geq \lambda_{I_j}+\lambda_{J_i}.\]
    Which implies $\lambda_{I_j}=0$. Hence $\lambda_I=0$, for all $I\in P$. Thus the linear combination lies in $\omega_Q^\vee$, which contains no linear subspace, so all coefficients must vanish. This completes the proof of statement (3).
\end{proof}
\subsection{Maximal bunches and complexes.}
Recall the definition of a maximal bunch of orbit cones (see Definition \ref{def:bunch}). Here we introduce maximally-biconnected complexes and show they parameterise the maximal bunches of orbit cones consisting of orbit cones in $Y^\free$. We show that the number of maximally-biconnected complexes on $[n]$ equals the $n$-th Hoşten-Morris number.
\begin{definition}\label{def:complex} \label{def:complextobunch}
	A \emph{complex} (on $[n]$) is a set $\Delta$ of non-empty subsets $I\subseteq [n]$, such that if $J\subseteq I$ and $I\in \Delta$, then $J\in \Delta$.
    A complex $\Delta$ is \emph{biconnected} if for all $I,J\in \Delta$ then $I\cup J \neq [n]$. A complex $\Delta$ is \emph{maximally-biconnected} if it is biconnected and not contained in any larger biconnected complex. A complex $\Delta$ is \emph{full} if $\{i\}\in \Delta$, for all $i\in[n]$.

    Let $\Delta$ be a complex. We associate a collection of cones \[\Phi_\Delta\vcentcolon=\{\omega_P\in\Omega_Y^\free \ \vert \ P\subseteq\Delta\}\subseteq\Omega_Y^\free.\]
\end{definition}

\begin{proposition}\label{prop:maxbunchbij}
	Let $\Delta$ be a full maximally-biconnected complex. Then $\Phi_\Delta$ is a maximal bunch of orbit cones. The induced map
	\begin{align*} \left\{\begin{array}{c}
        \text{Full maximally-biconnected} \\
        \text{complexes}
      \end{array}\right\} &\longrightarrow \left\{\begin{array}{c} \text{Maximal bunches of orbit cones } \\ \Phi\subseteq\Omega_Y^\free\end{array}\right\}, \\
    \Delta &\longmapsto \Phi_\Delta,\end{align*}
	is a bijection.
\end{proposition}
\begin{proof}
    We begin by showing that $\Phi_\Delta$ is a maximal bunch of orbit cones. Let $P\subseteq \Delta$ be a partition of $[n]$, by definition $\omega_P\in \Phi_\Delta$. Suppose $\omega\in \Omega_Y$ such that $\omega_P^\circ\subseteq\omega^\circ$. This implies that $\omega$ is top-dimensional, hence $\omega\in \Omega^\free_Y$ and $\omega=\omega_Q$ for some partition $Q$ of $[n]$. Since $\omega_P$ and $\omega_Q$ are both top dimensional we have that $\omega_P\subseteq \omega_Q$, thus by Lemma \ref{lem:twoprop}, $Q$ is a refinement of $P$. Since $\Delta$ is closed under taking subsets we have that $Q\subseteq \Delta$, therefore $\omega_Q\in\Phi_\Delta$. Hence $\Phi_\Delta$ satisfies Definition \ref{def:bunch} (2) in the definition of a bunch of orbit cones.

    To show that $\Phi_\Delta$ satisfies Definition \ref{def:bunch} (1), let $\omega_P,\omega_Q\in \Omega_Y^\free$ such that $\omega_P^\circ\cap\omega_Q^\circ=\emptyset$. By Lemma \ref{lem:twoprop}, there exists some $I\subseteq[n]$ such that $\omega_P\subseteq \eta_I$ and $\omega_Q\subseteq \eta_{I^c}$. If $\omega_P,\omega_Q\in\Phi_\Delta$, then by definition we have that $I,I^c\in \Delta$, which is absurd since $\Delta$ is biconnected. Therefore for all $\omega_P,\omega_Q\in \Delta$ we have that $\omega_P^\circ\cap\omega_Q^\circ\neq\emptyset$, thus $\Phi_\Delta$ is a bunch of orbit cones. 

    We claim $\Phi_\Delta$ is a \emph{maximal} bunch of orbit cones. Suppose that $\Phi$ is a bunch of orbit cones containing $\Phi_\Delta$. Let $\omega_P\in \Phi$. We first treat the case that $\omega_P$ is not top dimensional; the following two subcases can occur:
    \begin{enumerate}
        \itemsep0.2em
        \item $P$ is a partition of $I$, with $I\neq[n]$.
        \item $P$ is a partition of $[n]$, with $\#P\leq 2$.
    \end{enumerate}
    In the first case $\omega_P^\circ\cap\omega_Q^\circ=\emptyset$, for all $\omega_Q\in\Omega_Y^\free$, since $\omega_{P}$ is contained in the coordinate hyperplanes given by the vanishing of the $i$-th coordainte, for all $i\in I^c$. In the second case we may assume $\#P=2$, otherwise $\omega_P=\{0\}$. In this case $\omega_P=\eta_I\cap\eta_{I^c}$, for some $I\subseteq [n]$. It is a property of maximally-biconnected complexes that they contain precisely one of $I$ or $I^c$, for all $I\subseteq [n]$. Without loss of generality suppose $I\in \Delta$, then $\eta\in \Phi_\Delta$, and $\eta_I^\circ\cap\omega_P^\circ=\emptyset$, giving a contradiction. Therefore $\omega_P$ must be top dimensional, i.e. $\omega_P\in \Omega_Y^\free$.
    
    Let $\Delta_P$ be the complex consisting of the partition $P$ and its subsets. Since $\Phi$ is a bunch of orbit cones contains $\Delta_P$ by Definition \ref{def:bunch} (2). The union $\Delta\cup \Delta_P$ remains a full complex. Suppose $\{I,I^c\}\subseteq \Delta\cup \Delta_P$, then $\eta_I,\eta_{I^c}\in \Phi$, which cannot occur since $\eta_I^\circ\cap\eta_{I^c}^\circ=\emptyset$. Thus $\Delta\cup \Delta_P$ is a full biconnected complex. It follows, by maximality of $\Delta$, that $\Delta\cup \Delta_P=\Delta$, and hence $P\subseteq \Delta$ and $\omega_P\in \Phi_\Delta$. Thus $\Phi_\Delta$ is a maximal bunch of orbit cones.

    Finally we show the map is a bijection. Distinct maximally-biconnected complexes contain distinct partitions, hence their images contain distinct orbit cones, thus the map is injective. To show surjectivity, let $\Phi\subseteq \Omega_Y^\free$ be a maximal bunch of orbit cones. Let
	\[ \Delta_\Phi = \bigcup_{\omega_P\in \Phi}\Delta_P.\]
    Since $\Phi\neq\emptyset$ we must have $\omega_P\in \Phi$ for some partition $P$. Since $\Delta_P$ is full we have that $\Delta_\Phi$ is full. Again $\Delta_\Phi$ is biconnected since we cannot have that $\{\eta_I,\eta_{I^c}\}\subseteq \Phi$. 
    It is maximally-biconnected by the maximality of $\Phi$: the image of $\Delta_\Phi$ clearly contains $\Phi$, and by the maximality of $\Phi$ it is equal to it. Thus the map in the statement is a bijection.
\end{proof}

\begin{lemma}\label{lem:fullcomplexes}
    There are exactly $n$ maximally-biconnected complexes which are not full; these are the complexes generated by subsets of $\{i\}^c\subseteq[n]$, for some $i\in [n]$.
\end{lemma}
\begin{proof}
    Let $\Delta$ be a maximally-biconnected complex. Suppose $\{i\}\notin \Delta$. Then we must have $\{i\}^c\in \Delta$. Hence $\Delta$ consists entirely of the non-empty subsets of $\{i\}^c$. There are exactly $n$ such complexes.
\end{proof}
\begin{lemma}
    There is a bijection between maximally-biconnected complexes on $[n]$ and biconnected complexes on $[n-1]$.
\end{lemma}
\begin{proof}
    Let $I\subseteq[n]$. For the purposes of this proof let $C_n(I)$ denote the complement of $I$ within $[n]$. 
    
    Let $\Delta$ be a maximally-biconnected complex on $[n]$. Let
    \[ \widetilde{\Delta}\vcentcolon=\{C_{n-1}(I) \ \vert \ I\subseteq [n-1], \ I\notin \Delta\},\]
    where the complement of $I$ is taken within $[n-1]$. We claim this is a biconnected complex. Let $J_0\subseteq J_1\subseteq [n-1]$, with $J_1\in \widetilde{\Delta}$. Then $C_{n-1}(J_1)\notin \Delta$, hence $C_{n-1}(J_0)\notin \Delta$, since $\Delta$ is closed under taking subsets. Suppose $J_0,J_1\in \widetilde{\Delta}$, and $J_0\cup J_1=[n-1]$. Then $C_{n-1}(J_0),C_{n-1}(J_1)\notin \Delta$. Since $\Delta$ is maximally-biconnected we have that $C_n(C_{n-1}(J_0)),C_n(C_{n-1}(J_1))\in \Delta$. Since $n\in C_n(C_{n-1}(J_i))$ for each $i=0,1$ and $[n-1]\subseteq C_n(C_{n-1}(J_0))\cup C_n(C_{n-1}(J_1))$, this implies that $C_n(C_{n-1}(J_0))\cup C_n(C_{n-1}(J_1))=[n]$, contradicting the biconnectedness of $\Delta$.

    We now give an inverse construction. Let $\widetilde{\Delta}$ be a biconnected complex on $[n-1]$. Let
    \[ \Delta\vcentcolon= \{I\subseteq[n] \ \vert \ I\neq C_n(C_{n-1}(J)), \ \text{for any }J\subseteq[n-1], \ \text{such that } J\notin \widetilde{\Delta}\}.\]
    Suppose $I_0\subseteq I_1$, and $I_0=C_n(C_{n-1}(J))$ for some $J\subseteq[n-1]$ with $J\notin \widetilde{\Delta}$. Then $C_n(I_1)\subseteq C_n(I_0)\subseteq [n-1]$, and thus $C_{n-1}(C_n(I_0))=J\subseteq C_{n-1}(C_n(I_1))$. Since $J\notin\widetilde{\Delta}$ we have that $J'=C_{n-1}(C_n(I_1))\notin \widetilde{\Delta}$. Hence $I_1\notin\widetilde{\Delta}$. Thus $\Delta$ is a complex. Suppose $I_0,I_1\in \Delta$ and $I_0\cup I_1=[n]$. 
    
    Then $C_n(I_0)\cap C_n(I_1)=\emptyset$. Without loss of generality assume that $C_n(I_0)\subseteq [n-1]$. Then $C_{n-1}(C_n(I_0))\in\widetilde{\Delta}$ by definition.
\end{proof}
\begin{definition}[\cite{HM}, \cite{MillerSturmfels} Theorem 6.33]\label{def:hm}
    The \emph{Hoşten-Morris numbers} $\lambda(n)$ are the number of biconnected complexes on $[n-1]$. 
    
    These are known up to $n=9$, as given in the OEIS entry A001206 \cite{OEIS}.
    \medskip
    \begin{center}
        \begin{tabular}{l|ccccccl}\toprule
            \ \ $n$ & 5 & 6 & 7 & 8 & 9 \\
            $\lambda(n)$ & 81 & 2646 & 1422564 & 229809982112 & 423295099074735261880  \\ \bottomrule
        \end{tabular}
    \end{center}
    \medskip
\end{definition}
\begin{corollary}
    Consider the collection of $T$-invariant open subsets $U\subseteq Y$ admitting geometric quotients $U\rar U\git T$, such that the image $U\git T$ has the $A_2$-property. Let $\mathscr{G}$ denote the set of elements in this collection that are maximal with respect to $T$-equivariant inclusion. Then $\#\mathscr{G}=\lambda(n)-n$.
\end{corollary}
\begin{proof}
    By Proposition \ref{prop:3138} we have that the maximal bunches of orbit cones in $\Omega_Y^\free$ are in correspondence with $\mathscr{G}$. By proposition \ref{prop:maxbunchbij} these are in bijection with the full maximally-biconnected complexes. Lemma \ref{lem:fullcomplexes} then gives the count $\lambda(n)-n$.
\end{proof}
\subsection{All complete geometric quotients of $Y$.}
We show that given an arbitrary $T$-invariant open subset $U\subseteq Y$ admitting a geometric quotient $U\rar U\git T$, such that the image $U\git T$ is complete, then $U$ is obtained from the construction given in the previous subsection. In particular any such quotient space $U\git T$ has the $A_2$-property.
\begin{definition}
    Let $\mathscr{U}$ denote the set of $T$-invariant open subschemes $U\subseteq Y$ admitting a geometric quotient $U\rar U\git T$, such that the image $U\git T$ complete.
\end{definition}
\begin{theorem}\label{thm:polybij} There is a bijection
        \begin{align*}
            \{\text{Full maximally-biconnected complexes}\}&\longrightarrow \mathscr{U},\\
            \Delta &\longmapsto U(\Phi_\Delta)\git T.
        \end{align*}
    which factors through the bijection given in Proposition \ref{prop:maxbunchbij}
\end{theorem}
We utilise ideas developed by Proudfoot in \cite{Proudfoot}; our constructions often take place inside an ambient space obtained by gluing together all of the generic GIT quotients.
\begin{lemma} Let $y\in Y^\free$. There exists a generic $\theta\in \Theta$ such that $y$ is $\theta$-semistable, i.e.
        \[ Y^\free = \bigcup_{\substack{\theta\in \Theta \\ \mathrm{generic}}}Y^{\theta\text{-ss}}.\]
    \end{lemma}
    \begin{proof}
        Let $y\in Y^\free$. The orbit cone $\omega_y$ is top dimensional, thus it intersects the interior of some GIT chamber non-trivially. If $\theta$ lies in such a chamber then $y\in Y^{\theta\text{-s}}$. The opposite inclusion follows from the fact that a quiver representation with finite stabiliser has trivial stabiliser, see \cite[Lemma~1.3.2]{Ginzburg}.
    \end{proof}
\begin{definition}
    We define the \emph{total polygon space} $\calP_\tot$ as the quotient of $Y^\free$ by $T$.
    \[ \calP_\mathrm{tot}\vcentcolon=\bigcup_{\substack{\theta\in \Theta\\ \mathrm{generic}}}\calP_\theta = Y^\free\git T.\]
    Let $\pi_\calP$ denote the quotient morphism $Y^\free\rar\calP_\tot$, since $\pi_\calP$ is a geometric quotient it is both open and closed. The total space $\calP_\tot$ is a highly non-separated space, but it is covered by projective varieties, hence the structure morphism is universally closed.
\end{definition}
\begin{definition}
    For all $\omega=\omega_P\in \Omega_Y^\free$, we let 
    \[ V_\omega = \{y\in Y \ \vert \ \varphi_{i,j}(y)=0 \ \text{if and only if} \ i,j\in I, \ \text{for some} \ I\in P\}.\]
    This is a locally closed subscheme of $Y^\free$. Let $Z_\omega \vcentcolon= \pi_\calP(V_\omega)$. There is a commutative square
    \[\begin{tikzcd}
        {V_\omega} & {Y^\free} \\
        {Z_\omega} & \calP_\tot
        \arrow[from=1-1, to=1-2]
        \arrow[from=1-1, to=2-1]
        \arrow["{\pi_\calP}", from=1-2, to=2-2]
        \arrow[from=2-1, to=2-2]
    \end{tikzcd}\]
\end{definition}
    \begin{lemma}\label{prop:polystrata}
        Let $\omega\in \Omega_Y^\free$. The subsets $V_\omega\subseteq Y^\free$ and $Z_\omega \subseteq\calP_\tot$ are irreducible. Moreover the closures of $V_\omega$ and $Z_\omega$ are given by
        \begin{align*}
            \overline{V}_\omega =\bigcup_{\chi\subseteq \omega}V_{\chi}\subseteq Y^\free \ \text{and} \ \ \
            \overline{Z}_\omega =\bigcup_{\chi\subseteq \omega}Z_{\chi}\subseteq \calP_\tot.
        \end{align*}
    \end{lemma}
    \begin{proof}
        First note that
        \begin{align*}
            \overline{V}_\omega = \{y\in Y \ \vert \ \varphi_{i,j}(y)=0 \ \text{if} \ i,j\in I, \ \text{for some} \ I\in P\} = \bigcup_{\chi\subseteq \omega} V_\chi.
        \end{align*}
        Consider the closed subscheme  
        \[W\vcentcolon=\VV(x_i-x_j,y_i-y_j \ \vert \ i,j\in I, \ \text{for all }I\in P)\subseteq \calY.\] 
        Let $N=\#P$, then $V_\omega\cap W$ is an open subscheme of $W\cong \Af^{2N}$, hence it is irreducible.
        For each $I\in P$ choose a representative $i_I\in I$. We define a morphism $\alpha:V_\omega \rar V_\omega\cap W$ as follows. Let $\alpha(p_1,q_1,\ldots,p_n,q_n)= (p'_1,q'_1,\ldots,p'_n,q'_n)$, where $(p'_i,q'_i)\vcentcolon= (p_{i_I},q_{i_I})$, for $i\in I$. This morphism is a bundle over its image, with fiber given by the torus
        \[ \{(g_1,\ldots,g_n)\in T \ \vert \ g_{i_I}=1, \ \text{for all }I\in P\}\subseteq T.\] 
        Hence $V_\omega$ is irreducible. Since $\pi_\calP$ is open and closed, we obtain the analogous results for $Z_\omega$.
    \end{proof}
\begin{proposition}\label{prop:bunchquotientsarecomplete}
    Let $\Phi\subseteq\Omega_Y^\free$ be a maximal bunch of orbit cones. Then $U(\Phi)\git T$ is complete.
\end{proposition}
\begin{proof}
    The quotient $U(\Phi)\git T$ has the $A_2$-property by Proposition \ref{prop:3138}, hence it is separated. Therefore it suffices to show it is proper. Let $R$ be a valuation ring with fraction field $K$, together with a morphism $\alpha:\Spec(K)\rar U(\Phi)\git T\subseteq \calP_\tot$.  The image of $\alpha$ lies in $Z_\omega$, for some $\omega\in\Omega_Y^\free$. 
    
    For all generic $\theta\in\omega$, we have that $Z_\omega\subseteq \calP_\theta$, and since $\calP_\theta$ is projective, there is a unique extension $\beta_\theta:\Spec(R)\rar \calP_\theta$. Thus for all GIT chambers $C$ such that $\overline{C}\subseteq \omega$ there is a unique $\chi\in\Omega_Y^\free$, with $C\subseteq\chi_C\subseteq\omega$ such that, given $\theta\in C$, there is a unique extension $\Spec(R)\rar \calP_\theta$, with the image of the closed point lying in $Z_{\chi_C}$. The cones $\chi_C$ cover $\omega$, and we claim that their interiors have pairwise empty intersection. If their interiors did have non-empty intersection, then there would exist some generic GIT parameter in their intersection, which would contradict the properness of the resulting projective GIT quotient.

    Label the set of cones $\chi_C$ covering $\omega$ by $\chi_1,\ldots,\chi_N$. If $N=1$ then $\chi_1=\omega$, hence the morphism $\alpha$ admits an extension $\beta:\Spec(R)\rar U(\Phi)\git T$, where the closed point has image in $Z_\omega$, hence $Y\git_\Phi T$ is proper.

    Suppose $N\geq 2$. Take $i\neq j$, for $1\leq i,j\leq N$. Since $\chi_i^\circ\cap\chi_j^\circ=\emptyset$, by Lemma \ref{lem:twoprop} there exists some $I\subseteq[n]$ with $\chi_i\subseteq \eta_I$ and $\chi_j\subseteq\eta_{I^c}$. Since $\Phi=\Phi_\Delta$, where $\Delta$ is a full maximally-biconnected complex, we know that precisely one of $\eta_I$ or $\eta_{I^c}$ lies $\Phi$. Without loss of generality we assume $\eta_I\in \Phi$, hence $\chi_j\notin \Phi$. Consider the new collection of $N'$ cones consisting of $\chi_k$ such that $\chi_k\subseteq\eta_I$, we have that $N'<N$. Repeating this process until we are left with one cone $\chi_1=\omega_P$, for some partition $P$, we claim that $\omega_P\in \Phi$. By construction $\eta_I\in \Phi$, for all $I\in P$. Therefore $I\in \Delta$, for all $I\in P$. Thus by definition $\omega_P\in \Phi$. Since $\omega_P=\chi_1$, we have that there exists an extension $\alpha:\Spec(R)\rar U(\Phi)\git T$, where the closed point has image in $Z_{\chi_1}$. Hence $U(\Phi)\git T$ is complete.
\end{proof}
    We now complete the proof of the theorem.
\begin{proof}[Proof of theorem \ref{thm:polybij}]
    Note that by Proposition \ref{prop:bunchquotientsarecomplete} $U(\Phi_\Delta)\git T\in \mathscr{U}$, and the induced map is injective. Thus it suffices to prove surjectivity. Suppose $U\in \mathscr{U}$, we must have that $U\subseteq Y^\free$. Therefore $\pi_\calP(U)$ is a complete open subscheme of $\calP_\tot$.

    We begin by showing the following: let $\calP \subseteq \calP_\tot$ be a complete open subscheme, then there exists a unique full maximally-biconnected complex $\Delta$ such that \[\calP=\bigcup_{\omega\in \Phi_\Delta}Z_\omega.\]
    
    Let $p\in \calP$, then $p\in Z_\omega$ for some $\omega\in\Omega_Y^\free$. Since $\calP$ is open in $\calP_{\tot}$, we have that the generic point $\zeta_\omega$ of $Z_\omega$ lies in $\calP$. Let $q\in Z_\omega$, we claim that $q\in \calP$. 
    By \cite[\href{https://stacks.math.columbia.edu/tag/054F}{Lemma 054F}]{stacks-project}, there exists a discrete valuation ring $R$ with fraction field $K$ and a morphism $\alpha:\Spec(R)\rar \calP_\tot$ sending the open point to $\zeta_\omega$ and the closed point to $q$. We show this is the unique extension of the induced morphism $\Spec(K)\rar\calP_\tot$ to $\Spec(R)\rar \calP_\tot$. Since $\calP$ is proper, it follows that the image of the closed point $q\in\calP$, hence $\calP$ is a union of the strata $Z_\omega$. 
    
    Suppose there is an extension $\beta:\Spec(R)\rar\calP_\tot$ sending the closed point to $q'$. Since $\calP_\tot$ is covered by $\calP_\theta$, for $\theta\in\Theta$ generic, there exists some $\calP_\theta$ containing the image of $\Spec(R)$. Again since $\calP_\theta$ is open in $\calP_\tot$ we have that $\zeta_\omega\in \calP_\theta$ and since $\calP_\theta$ is a union of $Z_\chi$, we have that $q\in \calP_\theta$. Further since $\calP_\theta$ is projective, hence proper, we have that $q'=q$. Hence $\alpha=\beta$. 
    
    We have shown that $\calP=\bigcup_{\omega\in \Phi}Z_\omega$, for some subset $\Phi\subseteq\Omega_Y^\free$. We now show that $\Phi$ is a maximal bunch of orbit cones. By Proposition \ref{prop:polystrata}, we have that $\overline{Z}_\omega = \bigcup_{\chi \subseteq \omega}Z_\chi$, hence $\bigcup_{\omega\subsetneq \omega_0}Z_\chi$ is open. Since $\calP\subseteq \calP_\tot$ is an open subscheme, if $\omega\in \Phi$ and $\omega\subseteq\omega_0$ then $\omega_0\in \Phi$.
    If $\omega\in \Phi$ and $\omega^\circ\subseteq \omega_0^\circ$, then $\omega\subseteq\omega_0$ and thus $\omega_0\in \Phi$. Therefore $\Phi$ satisfies Definition \ref{def:bunch} (2).

    Let $\omega_P,\omega_Q\in \Omega_Y^\free$ such that $\omega_P^\circ\cap\omega_Q^\circ=\emptyset$, then by Lemma \ref{lem:twoprop} there exists some $I$ such that $\omega_P\subseteq\eta_I$ and $\omega_Q\subseteq\eta_{I^c}$. We claim that if $U\subseteq\calP_\tot$ is an open subscheme containing $Z_{\eta_I}\cup Z_{\eta_{I^c}}$, then $U$ is not separated. Since we have assumed $\calP$ is complete, this implies that we cannot have $\{\omega_P,\omega_Q\}\subseteq \Phi$, which implies that $\Phi$ satisfies Definition \ref{def:bunch} (1).

    We construct a morphism $\CC^*\rar \calP_\tot$, with two distinct extensions to a morphism $\Af^1\rar \calP_\tot$, such that $0$ has image in $Z_{\eta_I}$ and $Z_{\eta_{I^c}}$ respectively.
Let $\emptyset\neq I\subseteq[n]$. Up to reordering we may assume $I=\{1,\ldots,m\}$ and $I^c=\{m+1,\ldots,n\}$, with $1\leq m < n$. Consider the family 
\[v(t)=(p_1(t),q_1(t),\ldots,p_n(t),q_n(t))\in \calY^\free.\] For $t\neq 0$. Given by $(p_i(t),q_i(t))=(1,it)$, for $i\in I$, and $(p_j(t),q_j(t)) = (jt,1)$, for $j\in I^c$. Then
    \[ \varphi_{i,j}(v(t)) = \begin{cases}
        (j-i)t, \ &i,j\in I, \\
        1 - ijt, \ &i\in I, \ j\in I^c, \\
        (j-i)t, \ &i,j\in I^c.
    \end{cases}\]
    Define a family of elements $g(t)$ in $G$, given by
    \[ g_0(t) =\begin{pmatrix}
        1 & 0 \\ 0 & t
    \end{pmatrix}, \ g_i(t)=\begin{cases}
        1, & i\in I, \\
        t, & i\in I^c.
    \end{cases}\]
    We have $g(t)\cdot v(t) = (p_1'(t),q_1'(t),\ldots,p_n'(t),q_n'(t))$, where $(p_i'(t),q_i'(t))=(1,i)$, for $i\in I$, and $(p_i'(t),q_i'(t))=(it^2,1)$, for $i\in I^c$. We have $\lim_{t\rar 0}(g(t)\cdot v(t))\in V_{\eta_I}\subseteq Y^\free$.

    In addition define
    \[ h_0(t) =\begin{pmatrix}
        t & 0 \\ 0 & 1
    \end{pmatrix}, \ h_i(t)=\begin{cases}
        t, & i\in I, \\
        1, & i\in I^c.
    \end{cases}\]
    Then $h(t)\cdot v(t) =(p_1''(t),q''_1(t),\ldots,p_n''(t),q_n''(t))$, where $(p_i''(t),q''_i(t)) = (1,it^2)$, for $i\in I$, and $(p_i''(t),q''_i(t)) = (i,1)$, for $i\in I^c$. Thus we have $\lim_{t\rar 0}(h(t)\cdot v(t))\in V_{\eta_{I^c}}\subseteq Y^\free$. Therefore no separated subscheme $U\subseteq\calP_\tot$ contains $Z_{\eta_I}\cup Z_{\eta_{I^c}}$. Hence $\Phi$ is a bunch of orbit cones and since $\calP$ is complete it must be maximal.

    It follows that if $U\in \mathscr{U}$, then $\pi_\calP(U)=\bigcup_{\omega\in\Phi_\Delta}Z_\omega$. Therefore $U=\bigcup_{\omega\in\Phi_\Delta}V_\omega = U(\Phi)$, thus the map is surjective.
\end{proof}
\subsection{Examples.}
Here we describe some examples of singular polygon spaces and how their crepant resolutions can be obtained from Theorem \ref{thm:polybij}.
\begin{example}\label{ex:segre}
    Let $n=6$, and $\theta=(1,1,1,1,1,1)$. It is shown in \cite{HMSV} that $Z_3\vcentcolon=\calP_{\theta}$ is the Segre cubic, this is the projective 3-fold in $\PP^5$ defined by the equations
    \begin{align*}
        \sum_{i=0}^5 x_i = 0, \ \ \sum_{i=0}^5 x_i^3 =0.
    \end{align*}
    The variety $Z_3$ has 10 ordinary double points, each of which locally admits 2 crepant resolutions. Since each singular point can be resolved independently, $Z_3$ has $2^{10}$ crepant resolutions. By work of Finkelnberg \cite{Finkelnberg} only 332 of these are projective. We demonstrate that our results provide an alternative proof of this result. 
    
    Let $\PP^3_{5}$ be the blow-up of $\PP^3$ at 5 points in general position. This is a smooth variety which contains 10 disjoint $(-1,-1)$-curves. Contracting each of these curves to a point gives a morphism $f:\PP^3_5\rar Z_3$.
    Any crepant resolution of $Z_3$ is isomorphic in codimension 1 to $\PP^3_{5}$ (which is a Mori dream space by Lemma \ref{lem:polycox}). Therefore each projective crepant resolution of $Z_3$ corresponds to a Mori chamber in the movable cone of $\PP^3_5$. By Theorem \ref{thm:hukeel} these Mori chambers are identified with the GIT chambers associated to the torus action on $\Spec(\Cox(\PP^3_5))$ that contain the ray generated by $f^*L$ in their closure, where $L$ is an ample line bundle on $Z_3$. By Lemma \ref{lem:polygit}, the Mori chambers are determined by the hyperplane arrangement $\mathcal{A}$. Enumerating the chambers of the hyperplane arrangement $\mathcal{A}$ containing $(1,1,1,1,1,1)$ in their closure, one obtains the number 332.

    There are $10$ partitions $P_1,\ldots,P_{10}$ of $\{1,2,3,4,5,6\}$ into two subsets of size 3. Given $I_i\in P_i$ for all $i$, we construct a full maximally-biconnected complex generated by the subsets of $I_i$ for all $i$. Via the bijection given in Theorem \ref{thm:polybij} we obtain the $2^{10}$ crepant resolutions of $Z_3$ from these complexes.
    Let $\Delta$ be such a full maximally-biconnected complex. One can show that $U(\Phi_\Delta)\subseteq Y^{\theta\text{-ss}}$, hence there is an induced morphism $U(\Phi_\Delta)\git T\rar \calP_\theta$. Now note that $U(\Phi_\Delta)$ is covered by open subsets of $\calP_{\theta_+}$, where $\theta_+\in\Theta$ lies in a GIT chamber containing $\theta$ in its closure. Therefore the morphism $U(\Phi_\Delta)\git T \rar \calP_\theta$ is crepant.
\end{example}

\begin{example}\label{ex:segre2}
    Let $n=8$, and let $\theta = (1,1,1,1,1,1,1,1)$. Then $Z_5\vcentcolon=\calP_\theta$ is a 5-fold with 35 isolated ordinary double points. See \cite[Section~8.3.4]{projectiveinvariants} for a description of equations defining $Z_5$. Via Lemma \ref{lem:flipflop}, $Z_5$ is obtained from $\PP^5$ via the following birational transformations.
    \begin{enumerate}
        \item Select 7 distinguished points in $\PP^5$ in general position and blow them up.
        \item Flip the 21 disjoint curves given by the strict transform of lines spanned by 2 distinguished points in $\PP^5$.
        \item Contract the 35 disjoint planes given by the strict transform of planes spanned by 3 distinguished points in $\PP^5$.
    \end{enumerate}
    Each isolated ordinary double point admits two crepant resolutions related by a flop, thus $Z_5$ has $2^{35}$ minimal resolutions. These are obtain by considering all 35 partitions of $\{1,2,\ldots,7,8\}$ into two subsets of size 4, and associated $2^{35}$ full maximally-biconnected complexes similarly to the example of the Segre cubic.
    Counting the number of GIT chambers in $\Theta$ that contain the ray generated by $\theta$ in their closure, we find there are 495504 projective minimal resolutions of $Z_5$.
\end{example}
\begin{remark}
    The previous two examples generalise to a family. Let $n=2m$, with $m\geq 3$. Let $\theta=(1,1,\ldots,1)\in\Theta$. Then $Z_{n-3}\vcentcolon=\calP_\theta$ is a $n-3$ dimensional variety with $k\vcentcolon=\frac{1}{2}\binom{n}{m}$ ordinary double points. Thus $Z_{n-3}$ has $2^k$ crepant resolutions. 
    
    Let $\mathcal{B}$ be the collection of hyperplanes in $\QQ^{n-1}$ given by $\sum_{i\in I}z_i=0$, for all $I\subseteq[n-1]$ with $\#I =m$. The number of projective crepant resolutions of $Z_{n-3}$ is given by the number of chambers of the hyperplane arrangement $\mathcal{B}$.
\end{remark}
\section{All crepant resolutions of hyperpolygon spaces}
\subsection{The orbit cones of $X$.}
Recall the affine varieties $\calW$ and $\calX$ described in and above Definition \ref{def:calX}. We have \[\calW=\bigoplus_{i=1}^n \big(V^*(e_i)\oplus \det(V)(-2e_i)\big).\] 
Let $\pi_W:\calW \rar W:=\calW\git \SL(2)$ and $\pi_X:\calX\rar X:=\calX\git \SL(2)$ denote the affine quotients by $\SL(2)$. There are actions of $G$ on $\calW$ and $\calX$, and of $T$ on $W$ and $X$. Resulting in a commutative square
\[\begin{tikzcd}
	\calX & \calW \\
	X & Y
	\arrow[from=1-1, to=1-2]
	\arrow["{\pi_X}"', from=1-1, to=2-1]
	\arrow["{\pi_W}", from=1-2, to=2-2]
	\arrow[from=2-1, to=2-2]
\end{tikzcd}\]
Where the horizontal arrows are closed immersions equivariant with respect to $G$ and $T$ respectively. By Theorem \ref{thm:cox}, for $\theta\in \Theta$ generic there is a $T$-equivariant isomorphism $X\cong\Spec(\Cox(\MM_\theta))$. 

 Since $X$ is the spectrum of the Cox ring of a smooth irreducible variety, by Theorem \ref{thm:cox} and \cite[Theorem~1.5.1.1]{ADHL} $X$ is normal. Let $X^\free$ denote the open subscheme where $T$ acts freely on $X$. We begin by classifying the orbit cones in $\Omega_X$ and those in the subset $\Omega_X^\free=\{\omega_x \ \vert \ x\in X^\free\}$.
\begin{definition}
    Let $I,K\subseteq [n]$, and let $P$ be partition of $I$. We call such a collection \emph{orbit data}. Let $\epsilon_K=\cone(-e_k \ \vert \ k\in K)\subseteq \Theta$. Given orbit data $P,K$ let 
    \[\omega_{P,K}=\omega_{I,P}+\epsilon_K\subseteq \Theta.\]
\end{definition}
\begin{proposition}\label{prop:hyporbitcones}
    We have $\Omega_X \subseteq \{\omega_{P,K} \ \vert \ \text{Orbit data} \ P,K\}$, with $\omega_{P,K}\in \Omega_X$ if and only if one of the following conditions is satisfied:
    \begin{enumerate}
        \itemsep0.2em
        \item $\#\{J\in P \ \vert \ J\cap K\neq\emptyset\}\geq 4$.
        \item $\#(K \cap J) \neq 1$, for all $J\in P$.
    \end{enumerate}
    Moreover $\omega_{P,K}\in \Omega_X^\free$ if and only if $P$ is a partition of $[n]$, and one of the following conditions is satisfied:
    \begin{enumerate}
        \itemsep0.2em
        \item $\#\{J\in P \ \vert \ J\cap K\neq\emptyset\}\geq 4$.
        \item $\#K\cap J \neq 1$, for all $J\in P$, and $\#P \geq 3$.
        \item $\# K\cap J\neq 1$, for all $J\in P$, and $\#P\geq 2$, and $K\neq\emptyset$.
    \end{enumerate}
\end{proposition}
\begin{proof}
Since there is a $T$-equivariant closed immersion $X\subseteq W$, we have that $\Omega_X\subseteq\Omega_W$. We begin by characterising the orbit cones in $\Omega_W$.
The coordinate ring of $W$ is generated by the homogeneous elements $\varphi_{i,j}$ and $c_k$, for $1\leq i<j\leq n$ and $k\in [n]$. Given $w\in W$ we have
\[ \omega_w = \mathrm{Cone}\big(\deg(\varphi_{i,j}), \deg(c_k) \ \vert \ \varphi_{i,j}(w)\neq 0, \ c_k(w)\neq 0\big).\]

Since $\pi_W$ is surjective and determined by an injection on coordinate rings it suffices to check this vanishing on any point of the fiber $\pi_W^{-1}(w)$. 

Let $p=(p_1,q_1,r_1,\ldots,p_n,q_n,r_n)\in \calW$. Let $I=\{i \in [n] \ \vert \ (p_i,q_i)\neq 0\}$, and let $P$ be the partition of $I$ given by the equivalence relation $i\sim j$ if $p_iq_j-p_jq_i=0$. Let $K=\{k\in [n] \ \vert \ r_k\neq 0\}$. The orbit data $P,K$ determines exactly which generators of $\CC[W]$ vanish at $p$, hence at $\pi_W(p)$. Therefore we obtain 
\[\Omega_W=\{\omega_{P,K} \ \vert \ \text{Orbit data }P,K\}.\]
We claim that there exists $x\in X$ such that $\omega_x=\omega_{P,K}$ if and only if one of the following holds:
\begin{enumerate}
    \itemsep0.2em
    \item $\#\{J\in P\ \vert \ J\cap K\neq\emptyset\}\geq 4$.
    \item $\#J\cap K\geq 2$, for all $J\in P$.
\end{enumerate}
The subset $I\subseteq[n]$ and the partition $P$ depends only on the set of $(p_i,q_i)$, for $i\in [n]$. As in the proof of Lemma \ref{lem:polycones}, for any $I\subseteq [n]$ and partition $P$ of $I$ we can construct a collection of $(p_i,q_i)$ with this data. Given such a collection of $(p_i,q_i)$, for $i\in [n]$, and $K\subseteq[n]$, we can construct a point $p\in \calX$ such that $\omega_{\pi_X(p)}=\omega_{P,K}$ if and only if there exists $r_k\neq 0$, for $k\in K$, such that:
\begin{align*}
    \sum_{k\in K} (r_kp_k^2,r_kp_kq_k,r_kq_k^2)=(0,0,0)\in\CC^3.
\end{align*}

Some simple linear algebra tells us that we can find such a set of $r_k$ if and only if $(p_i^2,p_iq_i,q_i^2)\in \mathrm{Span}((p_j^2,p_jq_j,q_j^2) \ \vert \ j\in K\setminus\{i\})$, for all $i\in K$.
If $\#J\cap K\geq2$ for all $J\in P$, then this is satisfied, since removing any one vector has no effect on the total span. 
Similarly if $\#\{J \in P \ \vert \ J\cap K\neq\emptyset\}\geq 4$, then this is satisfied, since $(p_i^2,p_iq_i,q_i^2)$ lie in a 3-dimensional vector space and there are $\geq 4$ linear equivalence classes of vectors. 
Any vector of the form $(p_0^2,p_0q_0,q_0^2)$ lies in the cone over a smooth quadric in $\PP^2$, therefore no 3 pairwise linearly independent vectors of that form span a 2-dimensional space (since no 3 distinct points on a smooth quadric span a line).
Therefore if $\#\{J\in P \ \vert \ J\cap K\neq\emptyset\}=m$, with $m\leq 3$, then $\dim(\mathrm{span}((p_j^2,p_jq_j,q_j^2) \ \vert \ j\in K))=m$, and if there is some $J\in P$ with $J\cap K =1$, then removing the vector indexed by the element in $J\cap K$ leaves us with an $m-1$ dimensional span. 

This proves our claim, and completes the classification of $\Omega_X$. We now construct an open subset $\calW^\free\subseteq \calW$ with the property that the closed points of $\calW^\free$ are the closed points in $\calW$ with trivial stabiliser under the $G$-action. Hence we obtain a similar open subset \[\calX^\free\vcentcolon=\calW^\free\cap \calX\subseteq \calX.\]

Let $p=(p_1,q_1,r_1,\ldots,p_n,q_n,r_n)\in \calW$. Suppose $(p_i,q_i)=0$ for some $i\in [n]$. Then there is a subgroup $\ZZ/2\ZZ\subseteq G$, which acts by $-1$ at the vertex $i\in Q_0$, which stabilises $p$. Suppose that $(p_i,q_i)\neq 0$ for all $i$. Let $P$ be the partition of $[n]$ induced by $p$. We have that $[(g_0,g_1,\ldots,g_n)]\in G$ stabilises $p$ if and only if the following two statements hold
\begin{enumerate}
    \itemsep0.2em
    \item $g_0\cdot (p_i,q_i)  = (g_ip_i,g_iq_i)$, for all $i\in [n]$.
    \item $\det(g_0)r_i=g_i^2r_i$, for all $i\in [n]$.
\end{enumerate}
In a similar vein to the proof of Lemma \ref{lem:polycones}, one can show that $p\in\calW^\free$ if and only if one of the following conditions is satisfied
\begin{enumerate}
    \itemsep0.2em
    \item $\#P\geq 3$.
    \item $\#P\geq2$, and $K\neq\emptyset$.
\end{enumerate}
Thus we obtain an open subset $\calX^\free\subseteq \calX$.
We claim that $\pi_X^{-1}(X^\free)=\calX^\free$. As a representation of $\SL(2)$ we have that $\calW = (\oplus_{i=1}^n V^*)\oplus \CC^n$, let $\calW_0\cong\CC^n$ denote the trivial subrepresentation of $\calW$. We have that $\pi_W\vert_{\calW_0}$ is an isomorphism; denote its image by $W_0\subseteq W$. Note that $W_0\subseteq X$. It follows from \cite[Proposition~3.1]{GIT} that \[\pi_W:\pi_W^{-1}(W\setminus W_0)\rar W\setminus W_0.\] is a locally trivial $\SL(2)$-bundle. Since $\calW^\free\subseteq\calW\setminus\calW_0$, it follows that $\calW^\free=\pi_W^{-1}(W^\free)$, which implies the analogous result for $\calX^\free$.

Combining the conditions found for a point $p\in \calW$ to lie in $\calW^\free$, together with the conditions required for $\omega_{P,K}$ to lie in $\Omega_X$ we obtain the stated conditions required for $\omega_{P,K}$ to lie in $\Omega_X^\free$, this completes the proof.
\end{proof}
\begin{lemma}\label{lem:T+}
    Recall from Definition \ref{def:F} that $F\subseteq \Theta$ denotes the positive orthant with respect to a given basis. For all $\omega \in \Omega_X^\free$, we have that $\omega^\circ \cap  F^\circ\neq\emptyset$.
\end{lemma}
\begin{proof}
By Proposition \ref{prop:hyporbitcones}, we have $\omega=\omega_{P,K}$, with $\#P\geq 2$. Let $J_0$ and $J_1$ be distinct elements of $P$. For all $j\in [n]$, with $j\in J_0$, for some $i$, there is some $k\in [n]$, with $k\notin J_i$. Thus $e_j+e_k\in \omega$. The sum of such vectors for all $j\in [n]$ is a vector $v\in \omega$, for which each entry is positive. Hence $v\in  F^\circ$. Since $\omega$ is a top dimensional cone we have that $\omega^\circ \cap  F^\circ\neq\emptyset$.
\end{proof}
\subsection{All of the crepant resolutions of $\MM_0$.}
Let $\pi:\MM\rar \MM_0$ be a crepant resolution. By Lemma \ref{lem:semismall}, we have that for all $\theta\in \Theta$ generic, $\MM$ is isomorphic to $\MM_\theta$ in codimension 1. Since $\MM$ is smooth, hence normal, restriction and extension of sections induces a graded isomorphism 
\[\Cox(\MM)\cong \Cox(\MM_\theta).\]
\begin{definition}
    Let $\widehat{\MM}:=\Spec_{\MM}(\mathscr{C}ox(\MM))$. By Lemma \ref{lem:coxsheaf}, there is a $T$-equivariant open immersion $\widehat{\MM}\rar X$, with complement of codimension $\geq 2$.
\end{definition}
Recall from Definition \ref{def:hyperplanes} \[C_0 := \Big\{(\theta_i)\in \Theta \ \vert \ 0 \leq \theta_i \leq \sum_{j\neq i}\theta_j, \ \text{for all }i\in [n]\Big\}\subseteq \Theta.\] 
    For every $i\in [n]$, let
    \[ C_i = \Big\{(\theta_k)\in \Theta \ \vert \ \theta_k\geq 0, \ \sum_{j\neq i}\theta_j \leq \theta_i, \ \text{for all }k\in [n], \ \text{and }j\neq i\Big\}\subseteq \Theta.\]
    We have that $C_i=\cone(e_i,e_i+e_j \ \vert \ j\in\{i\}^c)$.
    It follows from the description of the GIT fan given in Theorem \ref{thm:BCSRW} that for $i\in [n]$, $C_i$ is a the closure of a single GIT chamber, while $C_0$ is the union of the closures of multiple GIT chambers. For $i\in [n]$, let $Q_i$ denote the partition $\{\{i\},\{i\}^c\}$. We have that $C_i = \omega_{Q_i,\{i\}^c}\cap F$.
\begin{lemma}\label{lem:cornercones}
    Let $\omega_{P,K}\in\Omega_X^\free$. Then $C_i\subseteq\omega_{P,K}$ if and only if $K\cap I^c\neq\emptyset$, for the unique $I\in P$ such that $i\in I$. Moreover, for all $\omega_{P,K}$ containing $C_i$, we have that $\omega_{Q_i,J}\subseteq \omega_{P,K}$, for some $J\subseteq\{i\}^c$ with $\#J\geq 2$.
\end{lemma}
\begin{proof}
     If $K\cap I^c\neq\emptyset$, then take $j_0\in K\cap I^c$. For all $j_1\in I$, we have that $e_{j_0}+e_{j_1}$ and $-e_{j_0}\in \omega_{P,K}$. Hence $e_i,e_i+e_{j_1}\in \omega_{P,K}$, for all $j_1\in I$. Further we have that $e_i+e_{j_1}\in\omega_{P,K}$ for all $j_1\in I^c$. Since $C_i=\cone(e_i,e_i+e_{j_1} \ \vert \ {j_1}\in\{m\}^c)$ we have $C_i\subseteq\omega_{P,K}$.

    Conversely if $K\cap I^c=\emptyset$, then for all $l\in I^c$ (which is not empty since $\#P\geq 2$) we have that in the sum
    \[ \sum_{j_1\not\sim_P j_2 }\lambda_{j_1,j_2}(e_{j_1}+e_{j_2})-\sum_{   k\in K}2e_k,\]
    the $l$-th coordinate is greater than or equal to the $i$-th coordinate, hence $e_i\notin \omega_{P,K}$. For the final statement, note that $\omega_{Q_i,J}=C_i+\epsilon_J$. So it suffices to show that $-e_{j}\in \omega_{P,K}$ for two distinct ${j}\in [n]\setminus\{m\}$, but the condition that $K\cap I^c\neq\emptyset$ guarantees the existence of such one $j$, and the conditions given in Proposition \ref{prop:hyporbitcones} for $\omega_{P,K}$ to lie in $\Omega_X^\free$ guarantee that $\#K\cap I^c \geq 2$.
\end{proof}
\begin{definition}\label{def:hypcomplextobunch}
    Let $\Delta$ be a maximally-biconnected complex. Let $\Phi_\Delta\subseteq\Omega_Y^\free$ denote the maximal bunch of orbit cones given in Definition \ref{def:complextobunch}. Let
\[ \Psi_\Delta = \begin{cases}
    \{\omega\in\Omega_X^\free \ \vert \ C_i\subseteq \omega\}, & \{i\}^c \in \Delta,\\
    \{\omega\in\Omega_X^\free \ \vert \ \omega_0\subseteq \omega, \ \text{for some }\omega_0\in \Phi_\Delta\}, & \Delta \ \text{full}. 
\end{cases} \]  
\end{definition}
\begin{lemma}
	Let $\Delta$ be a maximally-biconnected complex. Then $\Psi_\Delta$ is a maximal bunch of orbit cones.
\end{lemma}
\begin{proof}
    Suppose $\{i\}^c\in \Delta$. The cone $C_i^\circ$ is a GIT chamber. Given $\theta\in C_i^\circ$, we have that $X^{\theta\text{-ss}}\cong\widehat{\MM_\theta}$, on which $T$ acts freely. Therefore if $C_i\subseteq \omega$, then $\omega\in \Omega_X^\free$. The collection of orbit cones $\Psi=\{\omega\in\Omega_X^\free \ \vert \ C_i^\circ\subseteq\omega\}$ is then the maximal bunch of orbit cones associated to the generic GIT parameter $\theta$ (see Remark \ref{rmk:projcones}).

    Suppose $\Delta$ is full. Let $\omega\in \Psi_\Delta$, and $\chi\in \Omega_X$ such that $\omega^\circ \subseteq \chi^\circ$. Since $\omega$ is top-dimensional we have that $\omega\subseteq \chi$, and then $\chi\in \Psi_\Delta$ by definition, thus $\Psi_\Delta$ satisfies Definition \ref{def:bunch} (2). Given $\omega,\omega'\in \Psi_\Delta$, then there exists $\omega_0\subseteq \omega$ and $\omega_0'\subseteq\omega'$ with $\omega_0,\omega_0'\in \Phi_\Delta$. Since $\Phi_\Delta$ is a maximal bunch of orbit cones we have $\omega_0^\circ\cap(\omega_0')^\circ\neq\emptyset$, which implies $\omega^\circ\cap(\omega')^\circ\neq\emptyset$. Thus $\Psi_\Delta$ satisfies Definition \ref{def:bunch} (1).

    We claim that $\Psi_\Delta$ is a maximal bunch of orbit cones. Suppose $\Psi$ is a bunch of orbit cones such that $\Psi_\Delta\subseteq\Psi$. Let $\omega_{P,K}\in \Psi$. Since $C_0\in \Phi_\Delta$, we have that $C_0\in \Psi_\Delta$. We then have that $\omega_{P,K}^\circ\cap C_0^\circ\neq\emptyset$ by assumption. This implies that $\omega_{P,K}\in \Omega_X^\free$. Since $\omega_{P,K}\cap C_0 = \omega_P\in \Omega_Y^\free$, we have that $\omega_{P,K}^\circ \cap \chi^\circ = \omega_P^\circ \cap \chi^\circ$, for all $\chi\in \Phi_\Delta$, since $\chi\subseteq C_0$. Therefore $\omega_P\in \Phi_\Delta$, since $\Phi_\Delta$ is a maximal bunch of orbit cones, hence $\omega_{P,K}\in \Psi_\Delta$ by definition.
\end{proof}
\begin{theorem}\label{thm:B}
    Let $\mathscr{C}$ denote the set of crepant resolutions $\pi:\MM\rar \MM_0$ up to isomorphism over $\MM_0$.
    There is a bijection
    \begin{align*}
        \{\text{Maximally-biconnected complexes on }[n]\} &\longrightarrow \mathscr{C}, \\
        \Delta &\longmapsto \left[ U(\Psi_\Delta)\git T \rar \MM_0. \right]
    \end{align*}
\end{theorem}
\begin{remark}
    We note that unlike the polygon case, we do not claim that the maximally-biconnected complexes are in bijection with the maximal bunches of orbit cones in $\Omega_X^\free$, only that there is a bijection with the maximal bunches of orbit cones in $\Omega_X^\free$ admitting quotients proper over the base. There may or may not exist additional geometric quotients constructed via maximal bunches of orbit cones that are not proper over $\MM_0$.  
\end{remark}

The rest of this section is dedicated to a proof of Theorem \ref{thm:B}. 
\begin{lemma}\label{lem:hypfreelocus} The open subset $X^\free$ is covered by the $\theta$-semistable loci for $\theta\in F$ generic, that is 
    \[ X^\free = \bigcup_{\substack{\theta\in  F \\ \mathrm{generic}}}X^{\theta\text{-ss}}.\]
\end{lemma}
\begin{proof}
    Let $x\in X^\free $. By Lemma \ref{lem:T+} we have that $\omega_x^\circ\cap F^\circ\neq\emptyset$, thus there exists some generic $\theta$ in this intersection. Since a point $x\in X$ is $\theta$-semistable if and only if $\theta\in\omega_x$, it follows that $x\in X^{\theta\text{-ss}}$. Since $X^{\theta\text{-ss}}=\widehat{\MM}_\theta$ we have that $T$ acts freely at $x$.
\end{proof}
\begin{definition} We define the \emph{total hyperpolygon space} $\MM_\tot$ as the quotient of $X^\free$ by $T$.
    \[\MM_\mathrm{tot}:= \bigcup_{\substack{\theta\in  F \\ \mathrm{generic}}}\MM_\theta= X^\free\git T.\]
    Let $\Pi_\MM:X^\free \rar \MM_\tot$ denote the quotient morphism. Let $\pi_\tot:\MM_\tot \rar \MM_0$ denote the gluing of the natural morphism $\pi_\theta:\MM_\theta\rar \MM_0$.
\end{definition}
\begin{definition}
    Given $\omega=\omega_{P,K}\in\Omega_X^\free$, let 
    \[\mathcal{V}_\omega:=\left\{x\in X \ \Big\vert \ \begin{array}{c}\varphi_{i,j}(x)=0 \ \text{if and only if }i,j\in I, \ \text{for some }I\in P, \\ c_k(x)\neq 0, \ \text{if and only if }k\in K.\end{array}\right\}.\]
    Then $\mathcal{V}_\omega$ is a locally closed subscheme of $X^\free$
    and $\mathcal{Z}_\omega:=\Pi_\MM(\mathcal{V}_\omega)\subseteq \MM_\tot$.
    There is a commutative square
    \[\begin{tikzcd}
        {\mathcal{V}_\omega} & {X^\free} \\
        {\mathcal{Z}_\omega} & \MM_\tot
        \arrow[from=1-1, to=1-2]
        \arrow[from=1-1, to=2-1]
        \arrow["{\Pi_\MM}", from=1-2, to=2-2]
        \arrow[from=2-1, to=2-2]
    \end{tikzcd}\]
\end{definition}
\begin{lemma}\label{lem:addition}\label{lem:nullc}
    Suppose $K\neq\emptyset$, then $\omega_{P,K}\not\subseteq F$ if and only if $K\subseteq I$, for some $I\in P$. Moreover in this case we have 
    \[\omega_{P,K}\cap C_0 = \big\{z\in C_0 \ \vert \ \sum_{i\in I}z_i\leq \sum_{j\notin I}z_j\big\}.\]
\end{lemma}
\begin{proof}
    Given $\omega_{P,K}\in\Omega_X^\free$ such that $K\not\subseteq I$, for all $I\in P$. Let $I,I'\in P$ be distinct with $i\in I\cap K$ and $i'\in I\cap K$. Then for all $j\in I^c$ we have $e_i+e_j,-e_i\in \omega_{P,K}$, hence $e_j\in \omega$. Similarly for all $j\in (I')^c$ we have $e_j\in \omega$. Hence $F\subseteq \omega$.

    For the converse, first suppose that $F\subseteq \omega_{P,K}$, with $K\neq\emptyset$. Then $e_k$ and $-e_k$ lie in $\omega_{P,K}$, hence $\omega_{P,K}$ contains a 1-dimensional linear subspace. Suppose that $K\subseteq I$, for some $I\in P$. Consider the functional $L:\Theta \rar \QQ$ given by
    \[ L(z)= \sum_{j\notin I}2z_j - \sum_{i\in I} z_i.\] 
    We have that $\omega_{P,K}$ is generated by $e_i+e_j$, where $\{i,j\}\not\subseteq I$, together with $-2e_k$, for $k\in K$. Note that $L(e_i+e_j)>0$, when $\{i,j\}\not\subseteq I$, and $L(-2e_k)>0$ for all $k\in K$. Therefore $\omega_{P,K}$ cannot contain a 1-dimensional linear subspace, hence $F\not\subseteq \omega_{P,K}$.

    For the proof of the final statement we first consider the functional
    \[ L_i\vcentcolon=\sum_{j\notin I}f_j - \sum_{i\in I}f_i,\]
    and note that $L_i$ is non-negative on the generators of $\omega_{P,K}$. Therefore
    \[ \omega_{P,K}\subseteq \big\{z\in C_0 \ \vert \ \sum_{i\in I}z_i\leq \sum_{j\notin I}z_j\big\}.\]
    This implies $\omega_{P,K}\cap C_0 \subseteq \big\{z\in C_0 \ \vert \ \sum_{i\in I}z_i\leq \sum_{j\notin I}z_j\big\}$. For the converse inclusion let $z_0\in \big\{z\in C_0 \ \vert \ \sum_{i\in I}z_i\leq \sum_{j\notin I}z_j\big\}$. There is a decomposition $z_0=z_0'+z_0''$, where $z_0'\in \mathrm{Cone}(e_i+e_j \ \vert \ i\in I, j\notin I)$ and $f_i(z_0'')\geq 0$, for all $i\in I$. Thus both $z_0'$ and $z_0''$ lie in $\omega_{P,K}$.
\end{proof}
\begin{remark}
    Note that $\big\{z\in C_0 \ \vert \ \sum_{i\in I}z_i\leq \sum_{j\notin I}z_j\big\}=\eta_I$ (see Definition~\ref{def:etacone}). In particular if $K\subseteq I\in P$, then $\omega_{P,K}\cap C_0\in \Omega_Y^\free$ if and only if $\#I<n-1$.
\end{remark}
\begin{corollary}\label{cor:dda2}
    The zero fiber $\pi_\tot^{-1}(0)$ is proper over $\MM_0\setminus\{0\}$. Moreover
    \[\pi_\tot^{-1}(0)=\bigcup_{\omega\not\supseteq F}\calZ_\omega.\]
\end{corollary}
\begin{proof}
    Under the affine GIT quotient $X\rar \MM_0$, the fiber over $0\in\MM_0$, is the null-cone in $X$. These are precisely the points $x\in X$, such that $\omega_x$ is strongly convex, i.e. it contains no 1-dimensional linear subspace. Therefore $\pi_\tot^{-1}(0)$ is a union of the strata $\calZ_\omega$, where $\omega\in \Omega_X^\free$ is strongly convex.
    
    The proof of Lemma \ref{lem:addition} shows that, if $K\neq\emptyset$, then $\omega_{P,K}$ is strongly convex if and only if $F\not\subseteq \omega_{P,K}$. If $K=\emptyset$ then $\omega_{P,K}$ is strongly convex, hence we obtain the result.
\end{proof}
\begin{lemma}\label{lem:hyperpolystrata}
    Let $\omega\in\Omega_X^\free$, such that $F\not\subseteq \omega$ then $\mathcal{V}_\omega\subseteq X^\free $ and $\mathcal{Z}_\omega\subseteq\MM_\tot$ are irreducible, and their closures satisfy
    \begin{align*}
        \overline{\mathcal{V}}_\omega =\bigcup_{\chi\subseteq \omega}\mathcal{V}_{\chi}\subseteq X^\free , \ \text{and} \ \ \
        \overline{\mathcal{Z}}_\omega =\bigcup_{\chi\subseteq \omega}\mathcal{Z}_{\chi}\subseteq \MM_\tot.
    \end{align*}
\end{lemma}
\begin{proof}
    Suppose $\omega=\omega_{P,K}$, such that $ F\not\subseteq \omega$. By Lemma \ref{lem:nullc}, we have that $K\subseteq I$, for some $I\in P$. Let
    \[ \calS := \left\{p\in \calX \ \Big\vert \ \begin{array}{c}(p_i,q_i) = (p_j,q_j), \ i,j\in J, \ \text{for all }J\in P, \\ \text{and} \ r_k=0, \ \text{for all }k\notin K.\end{array}\right\}.\]
For each $J\in P$ select a representative $i_J\in J$. Define a morphism $\alpha: \mathcal{V}_\omega \rar \mathcal{V}_\omega \cap \calS$ by $\alpha(p)=(p_1',q_1',r_1',\ldots,p_n',q_n',r_n')$, where if $i\in I$ then $(p_i',q_i') = (p_{i_J},q_{i_J})$ and $r_i'= \lambda_i^2r_{i_J}$. Where $\lambda_i\in \CC^*$ is the unique element such that $\lambda_i(p_i,q_i) = (p_{i_J},q_{i_J})$. Note that
\[ \sum_{i\in [n]}(r_i'p_i'^2,r_i'p_i'q_i',r_i'q_i'^2)=\sum_{i\in [n]}(r_ip_i^2,r_ip_iq_i,r_iq_i^2)=(0,0,0).\]
The morphism $\alpha$ is a bundle for the torus
\[ \{(g_1,\ldots,g_n)\in T \ \vert \ g_{i_J}=1, \ \text{for all }J\in P\}\subseteq T.\]
The image $\mathcal{V}_\omega\cap\mathcal{S}$ is isomorphic to an open subset of $\Spec(\CC[x_{i_J},y_{i_J},c_{i} \ \vert \ J\in P])$ the subscheme determined by the equation $\sum_{i\in I}c_i=0$. In particular it is an open subset of an affine space, hence irreducible. Since $\mathcal{V}_\omega$ is a torus bundle over an irreducible subset it is irreducible. Apply $\Pi_\MM$ to $\mathcal{V}_\omega$ to obtain the result for $\mathcal{Z}_\omega$.
\end{proof}
\begin{proof}[Proof of Theorem \ref{thm:B}]
     Let $\pi:\MM\rar\MM_0$ be a crepant resolution. Then $\widehat{\MM}\subseteq X^\free$, hence $\MM\subseteq \MM_\tot$, proper over $\MM_0$. 
     
    \medskip
    
    \emph{Step 1}. We begin by showing that there exists some subset $\Psi\subseteq \Omega_X^\free$ such that 
    \begin{align*} \widehat{\MM} = \bigcup_{\Psi} \mathcal{V}_\omega\ \ \text{and} \ \ \MM = \bigcup_{\Psi} \mathcal{Z}_\omega. \end{align*}

    It follows from Lemma \ref{lem:codimen} that given a cartesian square
    \[\begin{tikzcd}
        {\mathfrak{U}} & {\mathfrak{M}_\tot} \\
        {\mathfrak{M}_0\setminus\{0\}} & {\mathfrak{M}_0}
        \arrow[from=1-1, to=1-2]
        \arrow[from=1-1, to=2-1]
        \arrow[from=1-2, to=2-2]
        \arrow[from=2-1, to=2-2]
    \end{tikzcd}\]
    then $\mathfrak{U}$ is separated, since $\mathfrak{U}\subseteq\MM_\theta$, for all generic $\theta\in F$. Since $\MM_\tot$ is universally closed over $\MM_0$ it follows that $\mathfrak{U}$ it is proper over $\MM_0\setminus\{0\}$. Since $\MM$ is proper over $\MM_0$, we must have that $\mathfrak{U}\subseteq \MM$. We have that $\mathfrak{U}$ consists of strata $\mathcal{Z}_\omega$. Thus to show $\MM$ is a union of strata $\mathcal{Z}$, it suffices to show that $\pi^{-1}(0)$ is a union of such strata.

    Suppose $p\in \mathcal{Z}_\omega$, where $\mathcal{Z}_\omega\subseteq \pi^{-1}(0)$. Let $\zeta_\omega$ be the generic point of $\mathcal{Z}_\omega$. Since $\MM$ is open in $\MM_\mathrm{tot}$, it is stable under generalisation, hence $\zeta_\omega\in \MM$. Let $q\in \mathcal{Z}_\omega$. By \cite[\href{https://stacks.math.columbia.edu/tag/054F}{Lemma 054F}]{stacks-project} there exists a discrete valuation ring $R$, with fraction field $K$, and a morphism $\alpha:\Spec(R)\rar \MM_\mathrm{tot}$ sending the open point to $\zeta_\omega$ and the closed point to $q$. 

    Consider the induced diagram
    \[\begin{tikzcd}
        {\Spec(K)} & {\MM_\tot} \\
        {\Spec(R)} & {\MM_0}
        \arrow[from=1-1, to=2-1]
        \arrow[from=2-1, to=2-2]
        \arrow[from=1-2, to=2-2]
        \arrow[from=1-1, to=1-2]
    \end{tikzcd}\]
    We will show that $\alpha$ is the unique extension $\Spec(R)\rar\MM_\tot$.
    
    Suppose $\Spec(R)\rar \MM_\tot$ is another extension sending the closed point to $q'$. Then since $\MM_\tot$ is covered by $\MM_\theta$, for $\theta\in F$ generic, there exists some $\MM_\theta$ containing $q'$ and $\zeta_\omega$. Since $\MM_\theta$ is a union of $\mathcal{Z}_\chi$ we have that $q\in \MM_\theta$, and since $\MM_\theta$ is proper over $\MM_0$, we have that the diagram
    \[\begin{tikzcd}
        {\Spec(K)} & {\MM_\theta} \\
        {\Spec(R)} & {\MM_0}
        \arrow[from=1-1, to=2-1]
        \arrow[from=2-1, to=2-2]
        \arrow[from=1-2, to=2-2]
        \arrow[from=1-1, to=1-2]
        \arrow[dashed, from=2-1, to=1-2]
    \end{tikzcd}\]
    admits a unique completion, hence $q'=q$. Therefore since $\MM$ is proper over $\MM_0$, it follows that $q\in \MM$. Thus $\pi^{-1}(0)$, and hence $\MM$, is a union of strata $\mathcal{Z}_\omega$. 

    \medskip

    \emph{Step 2}. we claim that $\Psi$ satisfies Definition \ref{def:bunch} (2). By Lemma \ref{lem:hyperpolystrata}, we have that $\overline{\mathcal{Z}}_\omega = \bigcup_{\chi \subseteq \omega}\mathcal{Z}_\chi$, hence $\bigcup_{\omega\subsetneq \omega_0}\mathcal{Z}_\chi$ is open. Since $\MM\subseteq \MM_\tot$ is an open subscheme, if $\omega\in \Psi$ and $\omega\subseteq\omega_0$ then $\omega_0\in \Psi$.
    If $\omega\in \Psi$ and $\omega^\circ\subseteq \omega_0^\circ$, then $\omega\subseteq\omega_0$ and thus $\omega_0\in \Psi$. Therefore $\Psi$ satisfies Definition \ref{def:bunch} (2).
    
    \medskip

    \emph{Step 3}. Given $\omega\in\Omega_X^\free$, let 
    $\rho(\omega)\vcentcolon=\omega\cap F$. Given $\Psi\subseteq \Omega_X^\free$, let
    \[ \rho(\Psi)\vcentcolon=\{\rho(\omega) \ \vert \ \omega\in \Psi\}.\]
    We claim that $C_i\in \rho(\Psi)$ for precisely one $i=0,\ldots,n$.

 By Lemma \ref{lem:T+}, $\rho(\omega)$ is a top dimensional cone in $\Theta$. Suppose that $\mathcal{Z}_\omega\subseteq\mathfrak{U}$, then for all generic $\theta\in F$, we have $\theta\in \omega$. Therefore $F\subseteq \omega$, and $\rho(\omega)=F$. In general $\rho(\omega)$ is not a cone in $\Omega_X$. Given $\omega\in\Omega_X^\free$, we claim that $\rho(\omega)=C_i$ if and only if $\omega=\omega_{Q_i,K}$, with $K\subseteq\{i\}^c$. This follows from the fact that $\omega_{P,K}\cap F$ is generated by $e_j$, where $j\in I\in P$, such that $K\cap I^c\neq\emptyset$, together with $e_{j_1}+e_{j_2}$, where $\{j_1,j_2\}\not\subseteq I$, for all $I\in P$.

Let $R$ be a discrete valuation ring with fraction field $K$, together with a morphism $\alpha:\Spec(R)\rar \MM_\tot$ sending the open point to the generic point of $\MM_\tot$, and the closed point to the generic point of $\mathcal{Z}_{C_0}$. We have an induced morphism $\Spec(K)\rar \MM$. By properness this admits a unique extension $\Spec(R)\rar \MM$ over $\MM_0$. If this extension has the same image as $\alpha$, then $C_0\in \Psi$, hence $C_0\in \rho(\Psi)$. If not, then the image lies in $\MM_\theta$, for some generic $\theta\in F$, such that $\theta\notin C_0$. Since $F=\bigcup_{i=0}^n C_i$, we have that $\theta\in C_i^\circ$, for some $i=1,\ldots,n$. Therefore the image of the closed point lies in $\mathcal{Z}_\omega$, such that $\omega\cap F=C_i$, and $C_i\in \rho(\Psi)$.

\medskip

\emph{Step 4}. Suppose $C_i\in \rho(\Psi)$ for $i\neq 0$, we claim that $\Psi=\Psi_\Delta$, where $\Delta$ is the unique maximally-biconnected complex such that $\{i\}^c\in\Delta$.

 Let $\omega\in \Psi$ such that $\rho(\omega)=C_i$. We have that $\omega=\omega_{Q_i,K}$, where $K\subseteq \{i\}^c$. Since $\omega_{Q_i,K}\subseteq \omega_{Q_i,\{i\}^c}$, we have that $\omega_{Q_i,\{i\}^c}\in \Psi$. Let $\omega_0=\omega_{Q_i,\{i\}^c}$. With respect to the open covering
    \[ \MM_\tot = \bigcup_{\substack{\theta\in  F \\ \mathrm{generic}}}\MM_\theta.\]
    the stratum $\mathcal{Z}_{\omega_0}$ is contained in a unique open subscheme given by $\MM_\theta$, where $\theta\in C_i^\circ$ (since $\omega_0$ contains a unique GIT chamber inside $ F$). In particular
    \[ \MM_\tot\setminus\bigcup_{\omega_0\subsetneq\chi}\mathcal{Z}_\chi \subseteq \MM_\theta.\]
    Therefore $\overline{\mathcal{Z}}_{\omega_0} \subseteq \MM$, which implies $\omega_{Q,J}\in \Psi$, for all $J\subseteq\{i\}^c$.

    By Lemma \ref{lem:cornercones}, together with Definition \ref{def:bunch} (2) we have shown that any cone $\omega$ containing the GIT chamber $C_i$ is contained in $\Psi$. Hence $\MM_\theta\subseteq \MM$, and by properness of $\MM$ we obtain equality. Thus $\Psi=\Psi_\Delta$, where $\Delta$ is the unique maximally-biconnected complex containing $\{i\}^c$.

    \medskip

    \emph{Step 5}. Suppose $C_0\in \rho(\Psi)$, we claim that $\Psi=\Psi_\Delta$, where $\Delta$ is a \emph{full} maximally-biconnected complex.

    We have seen that if $C_0\in \rho(\Psi)$, then $C_0\in \Psi$. In this case $\overline{\mathcal{Z}}_{C_0}\subseteq\MM$ is a complete geometric GIT quotient of some open subset of $Y$ (see Theorem \ref{thm:polybij}). We see this by noting that
    \[ \mathcal{V}_{C_0}=\{x\in X^\free  \ \vert \ c_k(x)=0, \ \text{for all }k, \ \varphi_{i,j}(x)\neq0, \ \text{for all }i,j\}.\]
    Which is a dense open subset of the closed subscheme $Y\subseteq X$ defined as the vanishing locus $\VV(c_i \ \vert \ i\in[n])$.
    Since $\overline{\mathcal{Z}}_{C_0} \rar \pi^{-1}(0)$ is a closed immersion of an irreducible variety into complete scheme, we obtain that $\overline{\mathcal{Z}}_{C_0}$ is a complete. Hence by Proposition \ref{prop:maxbunchbij} there is a full maximally-biconnected complex $\Delta$, such that $\overline{\mathcal{Z}}_{C_0}=U(\Phi_\Delta)\git T$. Therefore $\Phi_\Delta\subseteq \Psi$, since $\Psi$ satisfies Definition \ref{def:bunch} (2) and we have that $\Psi_\Delta \subseteq \Psi$.

    Let $\omega_{P,K}\in \Psi$. Then $\omega_{P,K}\cap C_0\neq\emptyset$, else $\omega_{P,K}\cap C_i\neq\emptyset$ for some $i=1,\ldots,n$, which implies $C_i\in \rho(\Psi)$, contradicting our assumption. Thus $\#P \geq 3$, and $\omega_{P,K}\cap C_0 = \omega_P$.

    Since $\omega_P\subseteq\omega_{P,K}$, for all generic $\theta\in F$, we have that $\mathcal{Z}_{\omega_P}\subset\MM_\theta$ implies that $\mathcal{Z}_{\omega_{P,K}}\subseteq\MM_\theta$. Thus let $\Spec(R)\rar \MM_\tot$, sending its open point to the generic point of $\mathcal{Z}_{\omega_{P,K}}$ and its closed point to the generic point of $\mathcal{Z}_{\omega_P}$, then the induced morphism $\Spec(K)\rar \MM_\tot$ admits a unique extension over $\MM_0$. Therefore by properness of $\MM$ we have that $\mathcal{Z}_{\omega_P}\subseteq \MM$, hence $\omega_P\in \Psi$. Again since $\mathcal{Z}_{\omega_P}$ is contained inside the polygon space contained in $\MM$, we have that $\omega_P\in \Phi_\Delta$, hence $\omega_{P,K}\in \Psi_\Delta$. Therefore $\Psi=\Psi_\Delta$.

    \medskip

    \emph{Step 6}. Given a maximally-biconnected complex $\Delta$, let $\MM_\Delta\vcentcolon=U(\Psi_\Delta)\git T$. To finish the proof it suffices to show that $\MM_\Delta\subseteq \MM_\tot$ together with its natural morphism $\pi_\Delta:~\MM_\Delta\rar\MM_0$ is a crepant resolution of $\MM_0$. The map in the statement of the Theorem is then well-defined, and by the previous steps it is surjective. We have seen that if $\{i\}^c\in \Delta$ then $\MM_\Delta=\MM_\theta$, for $\theta\in C_i^\circ$. Thus we may assume $\Delta$ is full.
    
    First note that $U(\Psi_\Delta)\git T$ is an open subscheme of $\MM_\tot$, and $\MM_\tot$ is covered by $\MM_\theta$, for $\theta\in F$ generic. Hence each point of $\MM_\Delta$ is smooth, and locally around each point the natural morphism $\pi_\Delta$ is crepant. Since $\MM_\Delta$ is separated (since it has the $A_2$-property), to show $\pi_\Delta$ is a crepant resolution it suffices to show that it satisfies the existence part of the valuative criterion for properness over $\MM_0$.

    By the definition of $\Psi_\Delta$, if $\omega\in \Omega_X^\free$ and $\omega\supseteq F$ then $\omega\in \Psi_\Delta$, since $\omega$ necessarily contains some $\omega_0\in \Phi_\Delta$. It follows that
    \[ \bigcup_{\omega\supseteq F}\calZ_\omega \subseteq \MM_\Delta.\]
    The open subscheme $\bigcup_{\omega\supseteq F}\calZ_\omega\subseteq\MM_\tot$ is precisely the preimage of $\MM_0\setminus\{0\}$ under $\pi_\tot$, and by Corollary \ref{cor:dda2} we have that it is proper over $\MM_0\setminus\{0\}$. Therefore to prove properness of $\MM_\Delta$ it suffices to verify the properness of $\pi_\Delta^{-1}(0)$.

    Let $R$ be a discrete valuation ring fraction field $K$. Suppose we are given a morphism $\alpha:\Spec(K)\rar \pi_\Delta^{-1}(0)$, with the image lying in $\calZ_\omega$. Since $\calZ_\omega\subseteq \pi_\Delta^{-1}(0)$, by Lemma \ref{lem:addition} and Corollary \ref{cor:dda2}, we have that $\omega=\omega_{P,K}$, where $K\subseteq I$, for some $I\in P$.

    For all generic $\theta\in \omega\cap F$, we have that $\calZ_\omega\subseteq \MM_\theta$, and since $\MM_\theta$ is projective, there is a unique extension of $\alpha$ to a morphism $\Spec(R)\rar \MM_\theta$. Thus for all GIT chambers $C$ such that $\overline{C}\subseteq \omega$, there is a unique $\chi_C\in \Omega_X^\free$, with $C\subseteq \chi_C \subseteq \omega$, such that, given $\theta\in C$, there is a unique extension $\beta_C:\Spec(R)\rar\MM_\theta$, with the image of the closed point lying in $\calZ_{\chi_C}$. The union of these cones contains $\omega$ and their interiors have pairwise empty intersections.

    Label by $\chi_1,\ldots,\chi_N$ the subset of cones $\chi_C$ such that $\chi_C^\circ\cap C_0^\circ\neq\emptyset$. We claim that some $\chi_i\in \Psi_\Delta$, in which case the image of some extension of $\alpha$ lies in $\MM_\Delta$, showing that $\MM_\Delta$ is proper.

    If $\omega=\omega_{P,K}$ and $K\subseteq I\in P$ then by Lemma \ref{lem:addition} we have that $\omega_{P,K}\cap C_0=\eta_I\in\Omega_Y^\free$. Moreover, by definition of $\Psi_\Delta$, we have that $\omega_{P,K}\supseteq \omega_0$, for some $\omega_0\in \Phi_\Delta$. Since $\omega_0\subseteq C_0$, it follows that $\omega_0\subseteq \eta_I$, hence $\eta_I\in \Phi_\Delta$, by Definition \ref{def:bunch} (2).
    Similarly we have that $\chi_i\cap C_0\in \Omega_Y^\free$. Note that the cones $\chi_i\cap C_0$ cover $\eta_I$, and their interiors have pairwise empty intersections. Hence by the same argument given in the proof of Proposition \ref{prop:bunchquotientsarecomplete}, we have that some $\chi_i\cap C_0\in \Phi_\Delta$. Therefore by Definition \ref{def:bunch} (2), applied to $\Psi_\Delta$, we have that $\chi_i\in \Psi_\Delta$. This shows $\MM_\Delta$ is proper.
\end{proof}

We have already seen that the number of maximally-biconnected complexes on $[n]$ equals $\lambda(n)$.
\begin{corollary}
    There are $\lambda(n)$ crepant resolutions of $\MM_0$ up to isomorphism over $\MM_0$.
\end{corollary}

\printbibliography
\end{document}